\documentclass{article}
\usepackage[english]{babel}
\usepackage{graphicx}
\usepackage{textcomp}
\usepackage{amssymb}
\usepackage{amsmath}
\usepackage{amsthm}
\usepackage{setspace}
\usepackage[margin=2.5cm]{geometry}
\usepackage{bm}
\usepackage{centernot}
\usepackage{relsize}
\usepackage{hyperref}
\usepackage{amsmath}
\usepackage{amssymb}
\usepackage{mathtools}
\usepackage{xcolor}
\DeclareFontFamily{U}{mathx}{}
\DeclareFontShape{U}{mathx}{m}{n}{
	<-> mathx10
}{}
\DeclareSymbolFont{mathx}{U}{mathx}{m}{n}
\DeclareMathAccent{\widecheck}{0}{mathx}{"71}
\newtheorem{proposition}{Proposition}

\newtheorem{remark}{Remark}

\newtheorem{lemma}{Lemma}

\makeatletter
\renewcommand*\env@matrix[1][*\c@MaxMatrixCols c]{%
  \hskip -\arraycolsep
  \let\@ifnextchar\new@ifnextchar
  \array{#1}}
\makeatother

\begin{document}

\title{Very weak solutions of the heat equation with anisotropically singular time-dependent diffusivity}
\author{Zhirayr Avetisyan, Zahra Keyshams, Monire Mikaeili Nia, Michael Ruzhansky}
\date{}
\maketitle
\begin{abstract}
We investigate the heat equation with a time-dependent, anisotropic, and potentially singular diffusivity tensor.
Since weak (in the Sobolev sense) or distributional solutions may not exist in this setting, we employ the framework of very weak solutions to establish the existence and uniqueness of solutions to the heat equation with singular, anisotropic, time-dependent diffusivity.
\end{abstract}
\section{Introduction}
The concept of a solution to a partial differential equation (PDE) has evolved remarkably in the last century or more, from the classical solutions which satisfied the equation in the normal sense (i.e., pointwise), to the weak solutions in sense of Sobolev, or the weak solutions in the sense of an optimization problem (e.g., Lagrangian), all the way to the distributional solutions in various spaces of distributions, ultradistributions or even Fourier hyperfunctions. While each of these approaches has its advantages, it appears that none of them is capable of handling PDEs with strongest (e.g., distributional) coefficients. The recently proposed very weak solutions seem to be a promising way of addressing such problems.

The concept of very weak solutions of PDEs with singular (i.e., non-smooth) coefficients goes back to the paper \cite{GaRu15} by C. Garetto and M. Ruzhansky, where the authors considered second-order weakly hyperbolic equations with singular time-dependent coefficients. While weak (in the sense of Sobolev) or distributional solutions may not exist, it is established that a certain regularization scheme, called very weak solution, always produces appropriately unique solutions to such equations. Moreover, in case the equation does admit distributional or ultra-distributional solutions, the very weak solutions coincide with them in the appropriate sense. The latter condition, called consistency, is the main advantage of the method of very weak solutions over another widely used approach to singular PDEs called the Colombeau algebras.

Since then, the subject of very weak solutions has gained much momentum, see \cite{AltybayNEW, ARST21a, ARST21b, ChRT22, Gar20, RY20} for more references. However, to our knowledge the only work where the heat equation with time-dependent singular coefficients is considered is \cite{ARST21}, where the authors address the heat equation with a singular potential. In the present paper, we study the existence and uniqueness of very weak solutions of the heat equation without a potential, but with a time-dependent, fully anisotropic thermal diffusivity tensor, which is allowed to be strongly singular. Sharp, non-gradual (i.e., singular) coefficients in PDEs appear in many physically relevant problems, where the medium consists of more than one phases with a sharp transmission interface. In particular, when a phase transition in the medium happens much faster than the diffusion process under consideration, the diffusivity coefficients can be effectively seen to be suffering discontinuities. Such a rapid phase transition may happen during the crystallization of a supercooled fluid, and if the resulting crystal is anisotropic, then a diffusion process in this medium will be described with the kind of equations we handle in this paper. Another interesting situation where such equations may be relevant is cosmology. Our best understanding of the history of early universe is that it has undergone phase transitions, which have dramatically changed its chemical and thermal properties. Moreover, many models of early universe assume that it was anisotropic, which could result in anisotropic phase transitions. Thus, any diffusion process that has happened in the early universe at cosmic scales may have to be described by the kind of equations discussed in this work. Readers interested in PDE-theoretical and Fourier-analytical aspects of anisotropic cosmological models may consult \cite{AvVe13} and references therein. 

The structure of this paper is as follows: Section 3 develops the classical theory for the considered anisotropic time-dependent heat equation. 
After introducing the fundamental kernel and establishing its structural properties, this section proves the classical existence and uniqueness result for solutions in the appropriate functional framework.
Section 4 shows stability for the solution maps. it is preparing the ground for the analysis of irregular data. 
Section 5 introduces the framework of very weak solutions, based on mollified operator families together with moderateness and negligibility conditions. 
Within this section we establish the main results of the paper: the existence, uniqueness, and consistency theorems for very weak solutions. 
The consistency part shows the relation between very weak and classical solutions in the regular case.

\textbf{Convention.} 
Throughout this paper,  $C$ is used to represent positive constants whose values may change at every occurrence, which are independent of the main involved parameters. We use the notation $A\lesssim_{\mu}B$ to indicate that there exists a constant $C_{\mu}$, depending only on the parameter $\mu$, such that
$A\leq C_{\mu} B.$
When the dependence on $\mu$
is not essential or is clear from the context, we may write $A\lesssim B$ to simplify the notation to mean $A\leq CB$, for a positive constant $C$.
\section{The classical theory}
In this section, we introduce a heat equation characterized by an anisotropic, smooth, time-dependent diffusivity. Unlike the standard heat equation with constant or isotropic diffusion, this model incorporates a diffusion matrix that varies with time and may have different effects in different spatial directions (anisotropy). Here, we present a classical approach to solving this equation. Now, consider the following partial differential equation:
\begin{equation}
\partial_tu(t,x)-\mathcal{L}_tu(t,x)=f(t,x),\label{heatEq}
\end{equation}
where $\{\mathcal{L}_t\}_{t\in\mathbb{R}_+}$ is a smooth family of partial differential operators given by
$$
\mathcal{L}_tv(x)=\sum_{i,j=1}^n a_{ij}(t)\partial^2_{x_ix_j}v(x),\quad\forall x\in\mathbb{R}^n,\quad\forall t\in\mathbb{R}_+,\quad\forall v\in C^\infty(\mathbb{R}^n),\quad n\in\mathbb{N}.
$$
In this section, $a\in C^\infty(\mathbb{R}_+,\mathrm{GL}(n)) \,\cap \, C([0,\infty),\mathrm{GL}(n))$ is a matrix-valued smooth function which is continuous at $t=0$, such that $a(t)$ is a symmetric positive definite matrix for all $t\in[0,\infty)$. We will use the inner product in $\mathbb{R}^n$ to write
$$
\mathcal{L}_t=\langle\partial_x,a(t)\partial_x\rangle,\quad\forall t\in\mathbb{R}_+,
$$
where $\partial_{x}$ stands for the gradient.
Denote
$$
\Delta\doteq\left\{(s,t)\in[0,+\infty)\times\mathbb{R}_+\;\vline\quad s<t\right\},
$$
then its closure in $\mathbb{R}^2$ will be
$$
\overline{\Delta}=\left\{(s,t)\in[0,+\infty)^2\;\vline\quad s\le t\right\}.
$$
Introduce the matrix-valued smooth function $A\in C^\infty(\mathbb{R}_+,\mathrm{GL}(n))$ by
$$
A(t)=\int\limits_0^ta(s)ds,\quad\forall t\in\mathbb{R}_+.
$$
It is clear that $A(t)-A(s)$ is a symmetric positive definite matrix for every $(s,t)\in\Delta$.

Define the function $\mathcal{W}\in C^\infty(\Delta, \mathcal{S}(\mathbb{R}^n)) $ by
$$
\mathcal{W}(x;s,t)\doteq\frac1{\sqrt{(4\pi)^n\det[A(t)-A(s)]}}e^{-\frac14\langle x,[A(t)-A(s)]^{-1}x\rangle},\quad\forall x\in\mathbb{R}^n,\quad\forall(s,t)\in\Delta.
$$
\begin{remark}\label{remarkfourierinverse} We have
\begin{align}\label{fourierinverse}
	\mathcal{W}(x;s,t) = \mathcal{F}^{-1}\left(e^{-\langle[A(t)-A(s)]\cdot, \cdot \rangle}\right)(x)=\widecheck{\left(e^{-\langle[A(t)-A(s)]\cdot, \cdot \rangle}\right)}(x),~~~\forall x\in \mathbb{R}^{n},~~\forall (s,t)\in \Delta,
\end{align}
where $\mathcal{F}^{-1}$ stands for the inverse Fourier transform. This is because, for a symmetric positive definite matrix $B$ and an arbitrary vector $b$ , one can verify through computation that
\begin{align}\label{formulaforinverse}
	\int\limits_{\mathbb{R}^n} e^{-\frac{1}{2}x^{\top}Bx+b^{\top} x} dx = \sqrt{\dfrac{(2\pi)^n}{\det B}} e^{\frac{1}{2}b^{\top}B^{-1}b}.
\end{align}	
	Here, we have
\begin{align*}
	\mathcal{F}^{-1}\left(e^{-\langle[A(t)-A(s)]\cdot, \cdot\rangle} \right)(x)= (2\pi)^{-n}\int\limits_{\mathbb{R}^n} e^{-\frac{1}{2}\xi^{\top}(2[A(t)-A(s)])\xi+i\langle x,\xi\rangle} d\xi,~~~\forall x\in \mathbb{R}^{n},~~~\forall (s,t)\in \Delta.
\end{align*}
We recall that the Euclidean inner product is denoted by $\langle a,b \rangle =a^{\top}b$ for column vectors $a,b\in\mathbb{C}^{n}$.
Therefore, by letting $B=2[A(t)-A(s)]$ and $b^{\top}=ix$ in \eqref{formulaforinverse} where $x\in\mathbb{R}^{n}$ is a column vector, we obtain the result in \eqref{fourierinverse}.
\end{remark}
\begin{proposition}\label{pr1}
The following properties hold:
\begin{itemize}
\item[1.]
$$
\int\limits_{\mathbb{R}^n}\mathcal{W}(x;s,t)dx=1,\quad\forall(s,t)\in\Delta;
$$

\item[2.]
$$
\partial_t\mathcal{W}(x;s,t)-\mathcal{L}_t\mathcal{W}(x;s,t)=0,\quad\forall x\in\mathbb{R}^n,\quad\forall(s,t)\in\Delta;
$$

\item[3.]
$$
\lim_{\epsilon\to0+}\mathcal{W}(\cdot;s,s+\epsilon)=\delta,\quad\forall s\in[0,+\infty);
$$

\item[4.] 
$$
\lim_{\epsilon\to0+}\mathcal{W}(\cdot; t-\epsilon,t)=\delta,\quad\forall t\in(0,+\infty).
$$
\end{itemize}
\end{proposition}
\begin{proof}
In order to prove part 1, consider the Gaussian integral as follows:
$$
\int\limits_{\mathbb{R}^n} e^{-\alpha y^{\top}y} dy = \left(\frac{\pi}{\alpha}\right)^{\frac{n}{2}}, ~~\alpha>0.
$$
Here, we have 
\begin{align}\label{gaussian}
	I &\doteq\int\limits_{\mathbb{R}^n} e^{-\frac{1}{4}\langle x, [A(t)-A(s)]^{-1}x\rangle} dx \notag\\
	&= \int\limits_{\mathbb{R}^n} e^{-\frac{1}{4}\langle [A(t)-A(s)]^{-\frac{1}{2}}x, [A(t)-A(s)]^{-\frac{1}{2}}x\rangle} dx,\quad\forall(s,t)\in\Delta.
\end{align}
Note that $[A(t)-A(s)]^{-\frac{1}{2}}$ exists because $A(t)-A(s)$ is a positive definite matrix. Let $y = [A(t)-A(s)]^{-\frac{1}{2}}x$, therefore, $dx = \det[(A(t)-A(s))^{\frac{1}{2}}] dy$.
Continuing the calculation in \eqref{gaussian}, we have
\begin{align*}
	I &= \int\limits_{\mathbb{R}^n}\sqrt{\det[A(t)-A(s)]}\, e^{-\frac{1}{4} y^Ty} dy  = (4\pi)^{\frac{n}{2}} \sqrt{\det[A(t)-A(s)]},\quad\forall(s,t)\in\Delta.
\end{align*}
Hence,
\begin{align*}
	\int\limits_{\mathbb{R}^n}\mathcal{W}(x;s,t)dx= \frac1{\sqrt{(4\pi)^n\det[A(t)-A(s)]}}\, \mathrm{I} =1,\quad\forall(s,t)\in\Delta.
\end{align*}
To prove part 2, note that
$$
\mathcal{L}_t\mathcal{W}(x;s,t)=\sum_{i,j=1}^n a_{ij}(t)\partial^2_{x_ix_j}\mathcal{W}(x;s,t),\quad\forall x\in\mathbb{R}^n,\quad\forall (s,t)\in\Delta.
$$
Therefore, we need to calculate $\partial^2_{x_ix_j}\mathcal{W}(x;s,t)$.
We have
\begin{align}
	\partial_{x_j}\mathcal{W}(x;s,t) &= \frac{1}{\sqrt{(4\pi)^n\det[A(t)-A(s)]}}  \partial_{x_j}\left(e^{-\frac{1}{4}\langle x, [A(t)-A(s)]^{-1}x\rangle}\right) \notag\\ 
	&= -\frac{1}{4} \frac{1}{\sqrt{(4\pi)^n\det[A(t)-A(s)]}} e^{-\frac{1}{4}\langle x, (A(t)-A(s))^{-1}x\rangle} 2 \left((A(t)-A(s))^{-1}x\right)_{j}\notag\\
	&= -\frac{1}{2}\mathcal{W}(x;s,t)\left((A(t)-A(s))^{-1}x\right)_{j},\quad\forall x\in\mathbb{R}^n,\quad\forall(s,t)\in\Delta.
\end{align}
Then, with a similar calculation,
\begin{align}
	&\partial^2_{x_ix_j}\mathcal{W}(x;s,t) = 
	\partial_{x_i}(\partial_{x_j}\mathcal{W}(x;s,t))= \partial_{x_i}\left(-\frac{1}{2}\mathcal{W}(x;s,t)\left((A(t)-A(s))^{-1}x\right)_{j}\right)\notag\\
	&= \frac{1}{4} \mathcal{W}(x;s,t) \left((A(t)-A(s))^{-1}x\right)_{i}\left((A(t)-A(s))^{-1}x\right)_{j} -\frac{1}{2} \mathcal{W}(x;s,t)\left((A(t)-A(s))^{-1}\right)_{i,j},\\
	&\forall x\in\mathbb{R}^n, \forall(s,t)\in\Delta.\notag
\end{align}
Therefore,
\begin{align}\label{LtW} 
	&\mathcal{L}_t\mathcal{W}(x;s,t)\notag\\
	&= \mathcal{W}(x;s,t)\left(-\frac{1}{2} \sum_{i,j=1}^n a_{i,j}(t)\left([A(t)-A(s)]^{-1}\right)_{i,j}+\frac{1}{4}\sum_{i,j=1}^n a_{i,j}(t)\left([A(t)-A(s)]^{-1}x\right)_{i}\left([A(t)-A(s)]^{-1}x\right)_{j}\right)\notag\\
	&=\mathcal{W}(x;s,t)\left(-\frac{1}{2} \operatorname{Tr}\left(a(t)[A(t)-A(s)]^{-1}\right)+\frac{1}{4}\langle[A(t)-A(s)]^{-1}x, a(t)[A(t)-A(s)]^{-1}x\rangle\right),\\
	&\forall x\in\mathbb{R}^n,\quad \forall(s,t)\in\Delta\notag.
\end{align}
To complete the proof, we need to calculate 
$\partial_t \mathcal{W}(x;s,t),$ for $ \forall x\in\mathbb{R}^n, \forall(s,t)\in\Delta$.
\begin{align}\label{partialtW}
	\partial_t \mathcal{W}(x;s,t)&=\partial_t\left(\frac{1}{\sqrt{(4\pi)^n\det[A(t)-A(s)]}}\right)e^{-\frac{1}{4}\langle x, [A(t)-A(s)]^{-1}x\rangle}\notag\\
	&+\frac{1}{\sqrt{(4\pi)^n\det[A(t)-A(s)]}}\partial_t\left(e^{-\frac{1}{4}\langle x, [A(t)-A(s)]^{-1}x\rangle}\right)\notag\\
	&=-\frac{1}{2\sqrt{(4\pi)^n}}{\det[A(t)-A(s)]^{-\frac{1}{2}}}\operatorname{Tr}\left([A(t)-A(s)]^{-1}a(t)\right)e^{-\frac{1}{4}\langle x, [A(t)-A(s)]^{-1}x\rangle}\notag\\
	&+\frac{1}{4}\frac{1}{\sqrt{(4\pi)^n\det[A(t)-A(s)]}}e^{-\frac{1}{4}\langle x, [A(t)-A(s)]^{-1}x\rangle}\langle x, [A(t)-A(s)]^{-1}a(t)[A(t)-A(s)]^{-1}x\rangle\notag\\
	&=\mathcal{W}(x; s,t)\left(-\frac{1}{2} \operatorname{Tr}\left([A(t)-A(s)]^{-1}a(t)\right)+\frac{1}{4}\langle [A(t)-A(s)]^{-1}x, a(t)[A(t)-A(s)]^{-1}x\rangle\right).
\end{align}
Comparing \eqref{LtW} and \eqref{partialtW} leads us to $$
\partial_t\mathcal{W}(x;s,t)-\mathcal{L}_t\mathcal{W}(x;s,t)=0,\quad\forall x\in\mathbb{R}^n,~ \forall(s,t)\in\Delta.
$$
To prove part $3$, note that for any $s\in [0,\infty)$, any $\epsilon>0$, and any $\phi \in \mathcal{S}(\mathbb{R}^n)$, the function  $\mathcal{W}(\cdot;s,s+\epsilon)$ belongs to $ \mathcal{S}(\mathbb{R}^n)\subset \mathcal{S}'(\mathbb{R}^n)$. Therefore, applying Remark \ref{remarkfourierinverse}, we obtain
\begin{align}
	\lim_{\epsilon\to0+}\mathcal{W}(\cdot;s,s+\epsilon)(\phi) &=\lim_{\epsilon\to0+} \int\limits_{\mathbb{R}^n} \mathcal{W}(x;s,s+\epsilon) \phi(x) dx\notag\\
	&=\lim_{\epsilon\to0+}\int\limits_{\mathbb{R}^n} \widecheck{\left(e^{-\langle [A(s+\epsilon)-A(s)]\xi, \xi \rangle}\right)}(x) \phi(x) \, dx
	\notag\\
	&=\lim_{\epsilon\to0+} \int\limits_{\mathbb{R}^n} \left(e^{-\langle[A(s+\epsilon)-A(s)]\xi, \xi \rangle}\right) \check{\phi}(\xi) d\xi \notag\\
	&=\int\limits_{\mathbb{R}^n} \lim_{\epsilon\to0+} \left(e^{-\langle[A(s+\epsilon)-A(s)]\xi, \xi \rangle}\right) \check{\phi}(\xi) d\xi \notag\\
	&=\int\limits_{\mathbb{R}^n} 1 \check{\phi}(\xi) d\xi= \check{1}(\phi) = \delta(\phi).
\end{align}
Note that by the Dominated Convergence Theorem, we can pass to the limit inside the integral because   $e^{-\langle[A(s+\epsilon)-A(s)]\cdot, \cdot \rangle}$  is a uniformly bounded function in $ \epsilon>0$ and $\forall s\in [0,\infty)$, as well as  $\check{\phi}\in \mathcal{S}(\mathbb{R}^n)\subset L^{1}(\mathbb{R}^n)$.

The proof of part 4 follows similarly to part 3. For every \( t > 0 \) and sufficiently small \( \epsilon > 0 \), we have \( \mathcal{W}(\cdot; t - \epsilon, t) \in \mathcal{S}(\mathbb{R}^n) \subset \mathcal{S}'(\mathbb{R}^n) \). Therefore, for any \( \phi \in \mathcal{S}(\mathbb{R}^n) \), by applying Remark \ref{remarkfourierinverse}, we obtain 
\begin{align}
	\lim_{\epsilon\to0+}\mathcal{W}(\cdot;t-\epsilon,t)(\phi) &=\lim_{\epsilon\to0+} \int\limits_{\mathbb{R}^n} \mathcal{W}(x;t-\epsilon,t) \phi(x) dx\notag\\
	&=\lim_{\epsilon\to0+}\int\limits_{\mathbb{R}^n} \widecheck{\left(e^{-\langle[A(t)-A(t-\epsilon)]\xi, \xi \rangle}\right)}(x) \phi(x) dx \notag\\
	&=\lim_{\epsilon\to0+} \int\limits_{\mathbb{R}^n} \left(e^{-\langle[A(t)-A(t-\epsilon)]\xi, \xi \rangle}\right) \check{\phi}(\xi) d\xi \notag\\
	&=\int\limits_{\mathbb{R}^n} \lim_{\epsilon\to0+} \left(e^{-\langle[A(t)-A(t-\epsilon)]\xi, \xi \rangle}\right) \check{\phi}(\xi) d\xi \notag\\
	&=\check{1}(\phi) = \delta(\phi).
	\end{align}
\end{proof}

Before proceeding, we first collect several remarks concerning the kernel $\mathcal{W}$, which will play an important role throughout the paper.
\begin{remark}\label{partialtimes}

For every $v\in\mathcal{S}(\mathbb{R}^n),~ (s,t)\in\Delta ~ \text{and} ~ \alpha, \beta \in \mathbb{N}_{0}$, we have 
\begin{align}
	\frac{\partial^{\alpha+\beta}}{\partial t^\alpha \partial s^{\beta}}\left(\mathcal{W}(\cdot; s,t)\ast v\right) = \left(	\frac{\partial^{\alpha+\beta}}{\partial t^\alpha \partial s^{\beta}}\mathcal{W}(\cdot; s,t)\right)\ast v,
\end{align}
since $\mathcal{W}\in C^\infty(\Delta, \mathcal{S}(\mathbb{R}^n))$ implies that the mapping 
$$
\Delta\longrightarrow \mathcal{S}(\mathbb{R}^n),
\quad
(s,t)\mapsto \mathcal{W}(\cdot;s,t),
$$ 
belongs to $ C^\infty(\Delta,\mathcal{S}(\mathbb{R}^n))$.
On the other hand, for fixed $v\in \mathcal{S}(\mathbb{R}^n)$, the convolution operator
$$
\mathcal{S}(\mathbb{R}^n)\to \mathcal{S}(\mathbb{R}^n), 
	\quad f \mapsto f * v,
$$
is linear and continuous. Hence, the composition
$$
\Delta\longrightarrow \mathcal{S}(\mathbb{R}^n),
\quad
(s,t)\longmapsto  \mathcal{W}(\cdot;s,t)*v
$$
belongs to $C^\infty(\Delta,\mathcal{S}(\mathbb{R}^n))$.
\end{remark}
\begin{remark}\label{partialtimesx}
For every $v\in\mathcal{S}(\mathbb{R}^n),~ (s,t)\in\Delta$ and $\beta \in \mathbb{N}^n_{0}$ as well as $x\in \mathbb{R}^n$, we have 
\begin{align}
{\partial^\beta_x}\left(\mathcal{W}(\cdot; s,t)\ast v\right)(x) = \left(	{\partial^\beta_x}\mathcal{W}(\cdot; s,t)\right)\ast v(x) = \mathcal{W}(\cdot; s,t)\ast\left({\partial^\beta_x}v\right)(x),
\end{align}
	Since $\mathcal{W}(\cdot;s,t)\in \mathcal{S}(\mathbb{R}^n)$ and $v\in \mathcal{S}(\mathbb{R}^n)$, their convolution belongs to $\mathcal{S}(\mathbb{R}^n)$.
\end{remark}
Here, we include the following lemma specifically related to the proof of Part 5 of Proposition \ref{pr2} stated below.

\begin{lemma}\label{upperboundderivative}
Let  $M$ be an $n \times n$  symmetric, positive definite, real-valued matrix where $n\in\mathbb{N}$. We denote the operator norm $\left\| M \right\|_{op}$ as
\[
\left\| M \right\|_{op} \doteq \sup_{|\xi|=1}|M\xi|,\quad\forall\xi\in\mathbb{R}^n,
\]
where $\left|M\xi \right|=\left( \sum^{n}_{i=1}\left( \sum^{n}_{j=1}M_{ij}\xi_{j}\right) ^{2}\right) ^{\frac{1}{2}}$.
Then, 
\begin{align}\label{moshtagh}
	\left|	\partial^{\gamma}_{\xi} e^{-\langle \xi, M\xi\rangle}\right|&\lesssim_{n,\gamma} ~~e^{ - \langle \xi, M \xi \rangle } (1 + |\xi|)^{|\gamma|} \sum_{\mu=1}^{|\gamma|}\|M\|_{op}^{\mu},\quad \forall \xi\in\mathbb{R}^{n},~~\forall \gamma\in\mathbb{N}^{n}_{0}.
\end{align}
\end{lemma}
\begin{proof}
To prove \eqref{moshtagh},
define functions $k$ and $g$ by
\begin{center}
\begin{minipage}{0.45\textwidth}
	\begin{align*}
		k &: [0, \infty) \to \mathbb{R}_+, \\
		&\quad k(y) \doteq e^{-y},\quad \forall y\in [0,\infty),
	\end{align*}
\end{minipage}
\hfill
\begin{minipage}{0.45\textwidth}
\begin{align*}
	g &: \mathbb{R}^n \to [0, \infty), \\
	&\quad g(\xi) \doteq \langle \xi, M \xi \rangle, \quad \forall \xi \in \mathbb{R}^n.
\end{align*}
\end{minipage}
\end{center}
Now, we use	Fa\`{a} di Bruno's formula, which takes the following form:
\begin{align}\label{faadibruno}
	\partial^{\gamma}_{\xi} k(g(\xi))&=	\sum_{\mu=1}^{|\gamma|} k^{(\mu)}(g(\xi)) \sum_{\substack{\beta_{1}+\ldots+\beta_{\mu}=\gamma \\ \beta_j\neq 0, ~ |\beta_j|\leq 2}} \frac{\gamma!}{\beta_{1}!\ldots\beta_{\mu}!}\prod_{j=1}^{\mu} \partial_{\xi}^{\beta_j}g(\xi),\quad \forall \xi\in\mathbb{R}^{n},~\forall \gamma\in\mathbb{N}^{n}_{0}.
\end{align}
If we suppose $\{e_{k}\}_{k=1}^n$ as the standard basis vector in $\mathbb{R}^{n}$, then we have 
\begin{equation}\label{derivative4}
k^{(\mu)}(g(\xi))=(-1)^{\mu}e^{-\langle \xi, M\xi \rangle},\quad\forall \xi\in\mathbb{R}^{n},
\end{equation} and
\begin{align}\label{dbetajg}
	\partial_{\xi}^{\beta_j}\langle \xi, M\xi \rangle=\begin{cases}
		2(M\xi)_{k}, & {} \beta_j=e_k,\\
		2M_{kl}, & {} \beta_j=e_k+e_l,\\
		0, & {otherwise},
	\end{cases}
\end{align}
where $j\in\{1,\ldots,\mu\}$ and $\mu\in\{1,2,\ldots,|\gamma|\}$ for $\gamma\in\mathbb{N}^{n}_{0}$.
It is worth noting about derivatives of  $g$ 
that 
\begin{align}\label{derivative1}
|\partial_{\xi}^{e_k} g(\xi)| = &\left|\frac{\partial}{\partial \xi_k} g(\xi)\right|= |2(M\xi)_k| = 2 \left|\sum_{j=1}^n M_{kj} \xi_j\right| \leq 2 \|M\|_{op} |\xi|<2 \|M\|_{op} (1+|\xi|),\\
&\forall\xi\in\mathbb{R}^{n},~~\forall k\in \{1,2,\ldots,n\}\notag,
\end{align}
and
\begin{equation}\label{derivative2}
|\partial_{\xi}^{e_k + e_\ell} g(\xi)| = \left|\frac{\partial^2}{\partial \xi_k \partial \xi_\ell} g(\xi)\right| = |2 M_{k\ell}| \leq 2 \|M\|_{op}<2 \|M\|_{op} (1+|\xi|),\quad \forall\xi\in\mathbb{R}^{n},~~\forall k,\ell\in \{1,2,\ldots,n\},
\end{equation}
and its higher derivatives vanish 
\begin{equation}\label{derivative33}\partial_{\xi}^\gamma g(\xi) = 0,\quad \forall\xi\in\mathbb{R}^{n},~~\forall \gamma\in\mathbb{N}^{n}_{0},~~ for ~ |\gamma|>2,
\end{equation}
because $g$ is quadratic.
Hence, since $\partial_{\xi}^\gamma g(\xi) = 0$ for all multi-indices $\gamma\in\mathbb{N}^{n}_{0}$ with $|\gamma| >2$, the only nonzero contributions in the Fa\`{a} di Bruno expansion arise from partitions of the form
$\beta_1+ \ldots+ \beta_{\mu}=\gamma$  where each  $\beta_{j}\in\mathbb{N}^{n}_{0}\setminus\{0\}$, $|\beta_j|\leq 2$, and $\mu\in\{1,2,\ldots,|\gamma|\}$.
On the other hand, considering \eqref{derivative1} and \eqref{derivative2}, we deduce that
\begin{equation}\label{derivative3}
\left| \prod_{j=1}^\mu \partial_{\xi}^{\beta_j} \langle \xi, M \xi \rangle \right| \leq 2^\mu \|M\|_{op}^\mu (1 + |\xi|)^\mu, \quad \forall \mu\in\mathbb{N},~~\forall\xi\in\mathbb{R}^{n},~~\beta_{j}\in\mathbb{N}^{n}_{0}\setminus\{0\},~~ |\beta_j|\leq 2.
\end{equation}
Consequently, considering \eqref{faadibruno}, \eqref{derivative4} and \eqref{derivative3}, we see
\begin{align*}
&\left| \partial_{\xi}^\gamma \left( e^{- \langle \xi, M \xi \rangle } \right) \right|\leq e^{-\langle \xi, M\xi \rangle}\sum_{\mu=1}^{|\gamma|}\left|  (-1)^{\mu} \sum_{\substack{\beta_{1}+\ldots+\beta_{\mu}=\gamma \\ \beta_j\neq 0, ~ |\beta_j|\leq 2}} \frac{\gamma!}{\beta_{1}!,\ldots\beta_{\mu}!}\prod_{j=1}^{\mu} \partial_{\xi}^{\beta_j}g(\xi)\right| \\
&\leq e^{-\langle \xi, M\xi \rangle}\sum_{\mu=1}^{|\gamma|} \sum_{\substack{\beta_{1}+\ldots+\beta_{\mu}=\gamma \\ \beta_j\neq 0, ~ |\beta_j|\leq 2}} \frac{\gamma!}{\beta_{1}!,\ldots\beta_{\mu}!}\prod_{j=1}^{\mu} \left|\partial_{\xi}^{\beta_j}g(\xi)\right|\\
&\leq e^{-\langle \xi, M\xi \rangle}\sum_{\mu=1}^{|\gamma|} \sum_{\substack{\beta_{1}+\ldots+\beta_{\mu}=\gamma \\ \beta_j\neq 0, ~ |\beta_j|\leq 2}} \frac{\gamma!}{\beta_{1}!,\ldots\beta_{\mu}!} 2^{\mu}\|M\|_{op}^{\mu} (1+|\xi|)^{\mu}\\
&=e^{-\langle \xi, M\xi \rangle}\sum_{\mu=1}^{|\gamma|}2^{\mu}\|M\|_{op}^{\mu} (1+|\xi|)^{\mu} \sum_{\substack{\beta_{1}+\ldots+\beta_{\mu}=\gamma \\ \beta_j\neq 0, ~ |\beta_j|\leq 2}} \frac{\gamma!}{\beta_{1}!,\ldots\beta_{\mu}!} \\
&\leq e^{-\langle \xi, M\xi \rangle}\sum_{\mu=1}^{|\gamma|}2^{\mu}\|M\|_{op}^{\mu} (1+|\xi|)^{\mu} \sum_{\substack{\beta_{1}+\ldots+\beta_{\mu}=\gamma \\ \beta_j\neq 0, ~ |\beta_j|\leq 2}} \gamma! \\
& \leq 2^{|\gamma|}\gamma! \left(\frac{n(n+3)}{2}\right)^{|\gamma|} e^{ - \langle \xi, M \xi \rangle } (1 + |\xi|)^{|\gamma|} \sum_{\mu=1}^{|\gamma|}\|M\|_{op}^{\mu}\notag\\
&\lesssim_{n,\gamma}e^{ - \langle \xi, M \xi \rangle } (1 + |\xi|)^{|\gamma|} \sum_{\mu=1}^{|\gamma|}\|M\|_{op}^{\mu},\quad \forall \xi\in\mathbb{R}^{n},~\forall \gamma\in \mathbb{N}^{n}_{0}.
\end{align*}
Thus, the proof is complete.
\end{proof}

Now, consider the family of linear operators $\{\operatorname{W}_{s,t}\}_{(s,t)\in\Delta}$, acting as
$$
\operatorname{W}_{s,t}:\mathcal{S}(\mathbb{R}^n)\to\mathcal{S}(\mathbb{R}^n),\quad\forall(s,t)\in\Delta,
$$

$$
\operatorname{W}_{s,t}v=\mathcal{W}(\cdot;s,t)\ast v,\quad\forall v\in\mathcal{S}(\mathbb{R}^n),\quad\forall(s,t)\in\Delta.
$$

We state Proposition \ref{pr2} concerning this family of operators. Before proceeding with this proposition we bring the following Remark for further use.
\begin{remark}\label{phi}
Fix $(s,t)\in\Delta$ and $\phi\in\mathcal{S}(\mathbb{R}^n)$. Since $\mathcal{W}(\cdot;s,t) \in \mathcal{S}(\mathbb{R}^n)$, the convolution $\mathcal{W}(\cdot;s,t) \ast \phi$ gives a Schwartz function and converges locally uniformly in $x$, together with all derivatives with respect to $x$. Therefore, for every $\alpha, \beta\in\mathbb{N}_0^n$, we have 
\begin{align}\label{c(s(R))}
	&\sup_{x\in \mathbb{R}^n}\left|x^\alpha \partial_x^\beta \operatorname{W}_{s,t}\phi(x)\right| =\sup_{x\in \mathbb{R}^n}\left|x^\alpha \partial_x^\beta (\phi\ast\mathcal{W}(\cdot; s,t)(x))\right|\notag\\
	&=\sup_{x\in \mathbb{R}^n} \left|x^\alpha\right|\left|\partial_x^\beta (\phi\ast\mathcal{W}(\cdot; s,t))(x)\right|=\sup_{x\in \mathbb{R}^n} \left|x^\alpha\right|\left| \int\limits_{\mathbb{R}^n}(\partial_x^\beta\phi(x-y))\mathcal{W}(y; s,t) dy\right| \notag\\
	&\leq \sup_{x \in \mathbb{R}^n} \int\limits_{\mathbb{R}^n} |x|^{|\alpha|} \left|\partial_x^\beta \phi(x-y)\right|\mathcal{W}(y; s,t)dy
	= \sup_{x \in \mathbb{R}^n} \int\limits_{\mathbb{R}^n} |x-y+y|^{|\alpha|} \left|\partial_x^\beta \phi(x-y)\right|\mathcal{W}(y; s,t)dy\notag\\
	&\leq 2^{|\alpha|} \sup_{x \in \mathbb{R}^n} \int\limits_{\mathbb{R}^n} |x-y|^{|\alpha|} \left|\partial_x^\beta \phi(x-y)\right|\mathcal{W}(y; s,t)dy+
	2^{|\alpha|} \sup_{x \in \mathbb{R}^n} \int\limits_{\mathbb{R}^n} |y|^{|\alpha|} \left|\partial_x^\beta \phi(x-y)\right|\mathcal{W}(y; s,t)dy\notag\\
	&\leq 2^{|\alpha|} \sup_{x \in \mathbb{R}^n} \int\limits_{\mathbb{R}^n} \left(1+|x-y|^2\right)^{|\alpha|} \left|\partial_x^\beta \phi(x-y)\right|\mathcal{W}(y; s,t)dy\notag\\
	&+	2^{|\alpha|} \sup_{x \in \mathbb{R}^n} \int\limits_{\mathbb{R}^n} \left|\partial_x^\beta \phi(x-y)\right|\left(1+|y|^2\right)^{|\alpha|}\mathcal{W}(y; s,t)dy, \quad\forall (s,t)\in \Delta.
\end{align}
Now, fix $\alpha, \beta \in \mathbb{N}_{0}^n$, and $\phi \in\mathcal{S}(\mathbb{R}^n)$. Define the following functions: 
\begin{align}\label{g}
	g: \mathbb{R}^n&\longrightarrow \mathbb{R},\notag\\
	g(x)&\doteq\left(1+|x|^2\right)^{|\alpha|} \left|\partial_x^\beta \phi(x)\right|,
\end{align}
and
\begin{align}\label{h}
	h: \mathbb{R}^n&\longrightarrow \mathbb{R},\notag\\
	h(x)&\doteq \left(1+|x|^2\right)^{|\alpha|}\mathcal{W}(x; s,t),\quad \forall(s,t)\in \Delta.
\end{align}
Using \eqref{g} and \eqref{h} in \eqref{c(s(R))}, we have
\begin{align}\label{convolution}
	\sup_{x\in \mathbb{R}^n} \left|x^\alpha \partial_x^\beta \operatorname{W}_{s,t}\phi(x)\right| &\leq 2^{|\alpha|} \sup_{x \in \mathbb{R}^n} \left(g\ast\mathcal{W}(\cdot; s,t)\right)(x)+ 2^{|\alpha|} \sup_{x \in \mathbb{R}^n} \left| \partial_x^\beta\phi\right|\ast h(x)\notag\\
	&=2^{|\alpha|} \left(\left\|g\ast \mathcal{W}(\cdot; s,t)\right\|_{\infty}+\left\||\partial_x^\beta\phi|\ast h\right\|_{\infty} \right)\notag\\
	&\leq 2^{|\alpha|} \left(\left\|g\right\|_{\infty} \left\|\mathcal{W}(\cdot; s,t)\right\|_{1}+\left\|\partial_x^\beta\phi\right\|_{\infty}\left\| h\right\|_{1} \right).
\end{align}
\end{remark}

With this preparation, we proceed with the following Proposition.

\begin{proposition}\label{pr2}
The following properties hold:
\begin{itemize}
\item[1.]
$$
\operatorname{W}_{s,t}\in C(\mathcal{S}(\mathbb{R}^n),\mathcal{S}(\mathbb{R}^n)),\quad\forall(s,t)\in\Delta;
$$

\item[2.] 
For every $v\in\mathcal{S}(\mathbb{R}^n)$, the map $(s,t)\mapsto\operatorname{W}_{s,t}v$ belongs to
$$
C^\infty(\Delta,\mathcal{S}(\mathbb{R}^n));
$$

\item[3.]
$$
\operatorname{W}_{r,s}\operatorname{W}_{s,t}=\operatorname{W}_{r,t},\quad\forall(r,s),(s,t)\in\Delta;
$$

\item[4.]
$$
\partial_t\operatorname{W}_{s,t}v-\mathcal{L}_t\operatorname{W}_{s,t}v=0,\quad\forall v\in\mathcal{S}(\mathbb{R}^n),\quad\forall(s,t)\in\Delta;
$$
\item[5.]
$$
\lim_{\substack{\delta,\epsilon \to 0^+ \\ s-\delta \ge 0}}\operatorname{W}_{s-\delta,s+\epsilon}v=v,\quad\forall v\in\mathcal{S}(\mathbb{R}^n),\quad\forall s\in(0,+\infty);
$$
\end{itemize}
\end{proposition}

\begin{proof}
To prove the first part, let $\{\phi_{m}\}_{m=1}^{\infty} \subset \mathcal{S}(\mathbb{R}^n)$ such that $\phi_{m} \xrightarrow{\mathcal{S}(\mathbb{R}^n)} 0$ when $m \rightarrow \infty$. It means that
\begin{equation}\label{assumption}
	\sup_{x \in \mathbb{R}^n} \left|x^\alpha \partial_x^\beta \phi_{m}(x)\right| \rightarrow 0, ~\text{as}\quad m\rightarrow \infty, \quad\forall\alpha, \beta \in \mathbb{N}_{0}^n.
\end{equation}
Now, we need to show that for every $\alpha, \beta \in \mathbb{N}_{0}^n$,
$$
\sup_{x \in \mathbb{R}^n} \left|x^\alpha \partial_x^\beta \operatorname{W}_{s,t}\phi_{m}(x)\right| \rightarrow 0, ~\text{as}\quad m\rightarrow \infty, ~~\forall(s,t)\in\Delta,~~ \forall\alpha, \beta \in \mathbb{N}_{0}^n.
$$
In this step, we make use of Remark \ref{phi}, where we set $\phi \doteq \phi_{m}$ and $g \doteq g_{m}$ for every $m\in\mathbb{N}$ to obtain
\begin{equation}
	\sup_{x\in \mathbb{R}^n} \left|x^\alpha \partial_x^\beta \operatorname{W}_{s,t}\phi_{m}(x)\right| \leq 2^{|\alpha|} \left(\left\|g_{m}\right\|_{\infty} \left\|\mathcal{W}(\cdot; s,t)\right\|_{1}+\left\|\partial_x^\beta\phi_{m}\right\|_{\infty}\left\| h\right\|_{1} \right).
\end{equation}	 
Therefore,
\begin{align*}
	\sup_{x\in\mathbb{R}^n} \left|x^\alpha \partial_x^\beta \operatorname{W}_{s,t}\phi_{m}(x)\right|\rightarrow 0,\quad \forall x\in \mathbb{R}^n,~~ \forall(s,t)\in\Delta,~~ \forall\alpha, \beta \in \mathbb{N}_{0}^n,~~\forall m\in \mathbb{N}.
\end{align*}
This is due to the facts that $\left\|\mathcal{W}(\cdot; s,t)\right\|_{1}<\infty, ~\text{and}~  \|h\|_{1} <\infty$, as well as the assumption \eqref{assumption}.
Hence, the proof of part $1$ is completed.

The proof of part 2 follows similarly to the reasoning in Remark \ref{partialtimes}.

To prove part 3, since 
$$
\operatorname{W}_{r,s}:\mathcal{S}(\mathbb{R}^n) \longrightarrow \mathcal{S}(\mathbb{R}^n), \quad \forall (r,s) \in \Delta,
$$
and
$$\operatorname{W}_{s,t}v \in \mathcal{S}(\mathbb{R}^n),~~\forall v\in \mathcal{S}(\mathbb{R}^n),~~\forall (s,t)\in \Delta,$$ then we can calculate $\operatorname{W}_{r,s}(\operatorname{W}_{s,t}v)$. Let $v\in \mathcal{S}(\mathbb{R}^n)$, then we have
\begin{align*}
	\widehat{\operatorname{W}_{r,s}(\operatorname{W}_{s,t}v)}&= \widehat{\mathcal{W}(\cdot;r,s)\ast \operatorname{W}_{s,t}v} = \widehat{\mathcal{W}(\cdot;r,s)} \widehat{\operatorname{W}_{s,t}v}\\ 
	&=\widehat{\mathcal{W}(\cdot;r,s)}(\widehat{ \mathcal{W}(\cdot;s,t)\ast v}) = \widehat{\mathcal{W}(\cdot;r,s)}\widehat{ \mathcal{W}(\cdot;s,t)}\hat{v},\\
	&\forall (s,t), (r,s)\in \Delta,~~\forall v\in \mathcal{S}(\mathbb{R}^{n}).
\end{align*}
Moreover,
\begin{align*}
	\widehat{\mathcal{W}(\cdot;r,s)}\widehat{ \mathcal{W}(\cdot;s,t)}(\xi) &= \widehat{\mathcal{W}(\xi;r,s)}\widehat{ \mathcal{W}(\xi;s,t)}\\
	&= e^{-\langle [A(s)-A(r)]\xi, \xi \rangle} e^{-\langle [A(t)-A(s)]\xi, \xi \rangle}\\
	&=e^{-\langle [A(t)-A(r)]\xi, \xi \rangle}=	\widehat{\mathcal{W}(\cdot;r,t)},\\
	&\forall (s,t), (r,s)\in \Delta.
\end{align*}
Therefore,
\begin{align*}
	\widehat{\operatorname{W}_{r,s}(\operatorname{W}_{s,t}v)}&= \widehat{\mathcal{W}(\cdot;r,s)}\widehat{ \mathcal{W}(\cdot;s,t)}\hat{v} = \widehat{\mathcal{W}(\cdot;r,t)}\hat{v}\\
	&=\widehat{\operatorname{W}_{r,t}v},\\
	&\forall (s,t), (r,s)\in \Delta,~~\forall v\in \mathcal{S}(\mathbb{R}^{n}).
\end{align*}
Thus,
$$
\operatorname{W}_{r,s}\operatorname{W}_{s,t}=\operatorname{W}_{r,t},~~~\forall (s,t), (r,s)\in \Delta.
$$
In part $4$, we want to prove that for every $v \in \mathcal{S}(\mathbb{R}^n)$ and for every $(s,t)\in \Delta$ we have 
$$
\partial_{t}\operatorname{W}_{s,t}v-\mathcal{L}_t\operatorname{W}_{s,t}v=0.
$$
Using the definition of the operator $\operatorname{W}_{s,t}$, we have
\begin{equation}\label{n1}
	\partial_{t}\operatorname{W}_{s,t}v-\mathcal{L}_t\operatorname{W}_{s,t}v = \partial_{t}\left(\mathcal{W}(\cdot;s,t)\ast v\right)-\mathcal{L}_t\left(\mathcal{W}(\cdot;s,t)\ast v\right),\quad \forall v \in \mathcal{S}(\mathbb{R}^n),\quad \forall(s,t)\in \Delta.
\end{equation}
From Remark \ref{partialtimes}, it follows that
\begin{equation}\label{n2}
	\partial_{t}\left(\mathcal{W}(\cdot;s,t)\ast v\right) = \partial_{t}\mathcal{W}(\cdot;s,t)\ast v,\quad \forall v \in \mathcal{S}(\mathbb{R}^n),\quad \forall(s,t)\in \Delta.
\end{equation}	
In addition, using Remark \ref{partialtimesx}, we have
\begin{align}\label{n3}
	&\mathcal{L}_{t}\left(\mathcal{W}(\cdot; s,t)\ast v\right)= \sum_{i,j=1}^n a_{i,j}(t) \frac{\partial^2}{\partial_{x_i}\partial_{x_j}}\left(\mathcal{W}(\cdot; s,t)\ast v \right)= \left(\sum_{i,j=1}^n a_{i,j}(t)\frac{\partial^2}{\partial_{x_i}\partial_{x_j}}\mathcal{W}(\cdot; s,t)\right)\ast v,\\
	&\quad  \forall v \in \mathcal{S}(\mathbb{R}^n),~~ \forall(s,t)\in \Delta.\notag 
\end{align}
Plugging \eqref{n2} and \eqref{n3} in \eqref{n1} we obtain the following:
\begin{align}
	\partial_{t}\operatorname{W}_{s,t}v-\mathcal{L}_t\operatorname{W}_{s,t}v&= \partial_{t}\mathcal{W}(\cdot;s,t)\ast v-\sum_{i,j=1}^n a_{i,j}(t)\frac{\partial^2}{\partial_{x_i}\partial_{x_j}}\mathcal{W}(\cdot; s,t)\ast v \notag\\
	&= \left(\partial_{t}\mathcal{W}(\cdot;s,t)-\sum_{i,j=1}^n a_{i,j}(t)\frac{\partial^2}{\partial_{x_i}\partial_{x_j}}\mathcal{W}(\cdot; s,t)\right)\ast v\notag\\
	&=\left(\partial_{t}\mathcal{W}(\cdot;s,t)-\mathcal{L}_{t}\mathcal{W}(\cdot; s,t)\right)\ast v
	=0, \quad  \forall v \in \mathcal{S}(\mathbb{R}^n),~~ (s,t)\in \Delta.
\end{align}
The justification for the final step in the above computation follows directly from part 2 of Proposition \ref{pr1}.

To prove part $5$, it suffices to show that for every $\hat{f}\in \mathcal{S}(\mathbb{R}^n)$, we have
\begin{equation}\label{Ws}
	\widehat{\operatorname{W}_{s-\delta,s+\epsilon}f}\xrightarrow{\mathcal{S}(\mathbb{R}^n)}\hat{f},\quad\text{when}\quad\epsilon,\delta\rightarrow 0,\quad\forall s\in(0,\infty),~~s-\delta\geq 0.
\end{equation}
Before proceeding with the proof, let us note some key points that will be used throughout the proof.
Firstly, for any $|\alpha|-$times differentiable complex-valued functions $f$ and $g$,
\begin{equation}\label{leibnitz}
	\partial^{\alpha}(fg)=\sum_{k\leq\alpha}\binom{\alpha}{k}
	\partial^kf \partial^{\alpha-k}g,~~\forall f,g\in C^{|\alpha|}(\mathbb{R}^n),~\forall \alpha\in\mathbb{N}_{0}^n.
\end{equation}
Secondly, in the following calculation, we denote the matrix norm $\left\| a(\tau) \right\|_{op}$ as
\[
\left\| a(\tau) \right\|_{op} \doteq \sup_{|\xi|=1}|a(\tau)\xi|, \quad \forall \tau \in [0,\infty),\quad\forall\xi\in\mathbb{R}^n,
\]

where $a\in C^\infty(\mathbb{R}_+,\mathrm{GL}(n))\cap C([0,\infty), \mathrm{GL}(n))$.

Now, applying \eqref{leibnitz}, we obtain the following:
\begin{align}\label{alpgabtaderivative}
	&\sup_{\xi\in\mathbb{R}^{n}}	\left|\xi^{\alpha}\partial_{\xi}^{\beta}\left(   \widehat{\operatorname{W}_{s-\delta,s+\epsilon}f(\xi)}-\hat{f}(\xi)\right) \right| =\sup_{\xi\in\mathbb{R}^{n}}\left| \xi^{\alpha} \partial_{\xi}^{\beta}\left(\widehat{\mathcal{W}(\cdot;s-\delta,s+\epsilon)}(\xi)\hat{ f}(\xi)-\hat{f}(\xi)\right)\right|  \notag\\
	&=\sup_{\xi\in\mathbb{R}^{n}}\left| \xi^{\alpha} \partial_{\xi}^{\beta}\left(\widehat{\left( \mathcal{W}(\cdot;s-\delta,s+\epsilon) -1\right) }\hat{f}(\xi)\right) \right|
	\leq\sup_{\xi\in\mathbb{R}^{n}}\sum_{\gamma\leq \beta}\binom{\beta}{\gamma}\left|\partial_{\xi}^{\gamma}\widehat{\left( \mathcal{W}(\xi;s-\delta,s+\epsilon) -1\right) }\right|\, \left| \xi^{\alpha}\partial_{\xi}^{\beta-\gamma}\hat{f}(\xi)\right|,\notag\\
	&\forall\hat{f}\in\mathcal{S}(\mathbb{R}^n),\quad \forall\alpha,\beta\in\mathbb{N}_{0}^n.
\end{align}
In view of \eqref{alpgabtaderivative}, we distinguish between two cases, $\gamma \neq 0$ and $\gamma = 0$. When $\gamma \neq 0$, it follows from Lemma \ref{upperboundderivative} that
\begin{align}\label{gamma}
	&	\left|\partial_{\xi}^{\gamma}\widehat{\left( \mathcal{W}(\xi;s-\delta,s+\epsilon) -1\right) }\right| =	\left|\partial_{\xi}^{\gamma}\widehat{ \mathcal{W}(\xi;s-\delta,s+\epsilon)}\right|\notag\\
	&\lesssim_{n,\gamma}e^{-\langle \xi, [A(s+\epsilon)-A(s-\delta)]\xi\rangle}(1 + |\xi|)^{|\gamma|} \sum_{\mu=1}^{|\gamma|}\left\| A(s+\epsilon)-A(s-\delta)\right\| _{op}^{\mu} \notag\\
	&\leq (1 + |\xi|)^{|\gamma|} \sum_{\mu=1}^{|\gamma|} (\epsilon+\delta)^{\mu}\left( \max_{\tau\in[0,s+1]}\|a(\tau)\|_{op}\right) ^\mu \notag\\
	&\leq (1 + |\xi|)^{|\gamma|} \sum_{\mu=1}^{|\gamma|} (\epsilon+\delta)^{\mu}\left(1+\max_{\tau\in[0,s+1]} \|a(\tau)\|_{op}\right) ^\mu \notag\\
	&\leq (1 + |\xi|)^{|\gamma|}\left( 1+\max_{\tau\in[0,s+1]}\|a(\tau)\|_{op}\right) ^{|\gamma|} \sum_{\mu=1}^{|\gamma|} (\epsilon+\delta)^{\mu}.
\end{align}
For sufficiently small values of $\epsilon,\delta$ such that $\epsilon+\delta < 1$, we have
$$
\sum_{\mu=1}^{|\gamma|} (\epsilon+\delta)^\mu
= (\epsilon+\delta)\,\frac{1-(\epsilon+\delta)^{|\gamma|}}{1-(\epsilon+\delta)}\le \frac{\epsilon+\delta}{1-(\epsilon+\delta)},\quad \forall \gamma \in \mathbb{N}^{n}.
$$
Consequently, the whole expression is finite and the estimate tends to zero as $\epsilon,\delta \to 0$.
Thus, we have
\begin{align}\label{epsilondelta1}
	&\sup_{\xi\in\mathbb{R}^{n}}	\left|\xi^{\alpha}\partial_{\xi}^{\beta}\left(   \widehat{\operatorname{W}_{s-\delta,s+\epsilon}f(\xi)}-\hat{f(\xi)}\right) \right| \leq \sup_{\xi\in\mathbb{R}^{n}}\sum_{\substack{\gamma\leq \beta\\\gamma\neq 0}}\binom{\beta}{\gamma}\left|\partial_{\xi}^{\gamma}\widehat{\left( \mathcal{W}(\xi;s-\delta,s+\epsilon) -1\right) }\left| \right| \xi^{\alpha}\partial_{\xi}^{\beta-\gamma}\hat{f(\xi)}\right|\notag\\
	&\lesssim_{n,\gamma} (1+\max_{\tau\in[0,s+1]}\|a(\tau)\|_{op})^{|\beta|} \sup_{\xi\in\mathbb{R}^{n}}\sum_{\substack{\gamma\leq \beta\\\gamma\neq 0}}\binom{\beta}{\gamma} (1 + |\xi|)^{|\gamma|}\left| \xi^{\alpha}\partial_{\xi}^{\beta-\gamma}\hat{f(\xi)}\right|\frac{\epsilon+\delta}{1-(\epsilon+\delta)}\notag\\
	&\lesssim_{n,\gamma, s,\alpha,\beta, f}\frac{\epsilon+\delta}{1-(\epsilon+\delta)},
\end{align}
where the last line is valid since $\hat{f}\in\mathcal{S}(\mathbb{R}^{n})$ and $a\in C^\infty(\mathbb{R}_+,\mathrm{GL}(n))\cap C([0,\infty), \mathrm{GL}(n))$.
We now turn to the case $\gamma=0$ as follows:
\begin{align}\label{eta0}
	&\sup_{\xi\in\mathbb{R}^{n}} 
	\left| \widehat{ \mathcal{W}(\xi; s-\delta, s+\epsilon) - 1 } \right| 
	\left| \xi^{\alpha} \partial_{\xi}^{\beta} \widehat{f}(\xi) \right| \notag\\
	&= \sup_{\xi\in\mathbb{R}^{n}} 
	\left| e^{-\langle \xi, [A(s+\epsilon) - A(s-\delta)] \xi \rangle} - 1 \right| 
	\left| \xi^{\alpha} \partial_{\xi}^{\beta} \widehat{f}(\xi) \right| \notag \\
	&= \sup_{\xi\in\mathbb{R}^{n}} 
	\left| \langle \xi, [A(s+\epsilon) - A(s-\delta)] \xi \rangle \right| 
	\cdot \frac{ \left| e^{-\langle \xi, [A(s+\epsilon) - A(s-\delta)] \xi \rangle} - 1 \right| }
	{ \left| \langle \xi, [A(s+\epsilon) - A(s-\delta)] \xi \rangle \right| }
	\left| \xi^{\alpha} \partial_{\xi}^{\beta} \widehat{f}(\xi) \right| \notag \\
	&\leq \left\| 
	\frac{ \left| e^{-\langle \cdot, [A(s+\epsilon) - A(s-\delta)] \cdot \rangle} - 1 \right| }
	{ \left| \langle \cdot, [A(s+\epsilon) - A(s-\delta)] \cdot \rangle \right| }
	\right\|_{\infty} \cdot 
	\sup_{\xi\in\mathbb{R}^{n}} |\xi|^2 \left\| A(s+\epsilon) - A(s-\delta) \right\|_{op}
	\left| \xi^{\alpha} \partial_{\xi}^{\beta} \widehat{f}(\xi) \right| \displaybreak\notag \\
	&\leq \left\| A(s+\epsilon) - A(s-\delta) \right\|_{op}
 \sup_{\xi\in\mathbb{R}^{n}} |\xi|^{|\alpha| + 2} 
	\left| \partial_{\xi}^{\beta} \widehat{f}(\xi) \right| \notag \\
	&\leq (\epsilon + \delta) \max_{\tau \in [0, s+1]} \| a(\tau) \|_{op}
 \sup_{\xi\in\mathbb{R}^{n}} |\xi|^{|\alpha| + 2} 
	\left| \partial_{\xi}^{\beta} \widehat{f}(\xi) \right| \notag \\
	&\lesssim_{s,\alpha,\beta, f} (\epsilon + \delta), \quad \forall \epsilon, \delta \in(0,1). 
\end{align}
It is important to note that the above calculations are justified by the following considerations:
\begin{equation}
	\langle \xi,[A(s+\epsilon)-A(s-\delta)]\xi\rangle>0,\quad \forall \xi\neq 0, ~\forall s\in(0,\infty),~~s-\delta\geq 0,
\end{equation}
and
$$
\left\|  \dfrac{1-e^{-\langle\cdot,[A(s+\epsilon)-A(s-\delta)]\cdot\rangle}}{\langle\cdot,[A(s+\epsilon)-A(s-\delta)]\cdot\rangle}\right\|_{\infty}=1,\quad \forall s\in(0,\infty),~ s-\delta\geq 0,
$$
and
$$
\left| \langle \xi,[A(s+\epsilon)-A(s-\delta)]\xi\rangle\right| \leq |\xi|^{2} \left\| A(s+\epsilon)-A(s-\delta)\right\|_{op},\quad\forall \xi\in\mathbb{R}^{n},~ \forall s\in(0,\infty),~ s-\delta\geq 0,
$$
and
$$
\left\| A(s+\epsilon)-A(s-\delta)\right\|_{op}\leq \int\limits^{s+\epsilon}_{s-\delta}\left\| a(\tau)\right\|_{op}d\tau\leq (\epsilon+\delta)\left( \max_{\tau \in[0,s+1]}\left\| a(\tau)\right\|_{op}\right) <\infty,$$
$$~~~~ \forall s\in(0,\infty),~ s-\delta\geq 0,
$$
as well as
$$
\sup_{\xi\in\mathbb{R}^{n}}\left|\xi\right|^{|\alpha|+2}\left| \partial^{\beta}\widehat{f}(\xi)\right|<\infty,\quad\forall \xi\in\mathbb{R}^{n},\quad\forall \alpha,\beta\in \mathbb{N}_{0}^{n},
$$  
given that $\widehat{f}\in \mathcal{S}(\mathbb{R}^n)$.
By substituting \eqref{epsilondelta1} and \eqref{eta0} into \eqref{alpgabtaderivative}, it follows that
\begin{align}\label{epsilon}
	&\sup_{\xi\in\mathbb{R}^{n}}\sum_{\gamma\le\beta}\binom{\beta}{\gamma}\left|\partial^{\gamma}\widehat{\left( \mathcal{W}(\xi;s-\delta,s+\epsilon) -1\right) }\right| \left| \xi^{\alpha}\partial^{\beta-\gamma}\widehat{f}(\xi)\right|\notag\\
	&\lesssim_{n,\gamma, s,\alpha,\beta, f}\frac{(\epsilon+\delta)^{|\gamma|}}{\epsilon+\delta-1}+(\epsilon+\delta).
\end{align}
Letting $\delta$ and $\epsilon$ tend to zero in \eqref{epsilon}, we obtain the desired result.
Thus, the proof of part 5 and, consequently, Proposition \ref{pr2}, is complete.
\end{proof}

The operator family \(\{\operatorname{W}_{s,t}\}_{(s,t) \in \Delta}\) allows us to solve the Cauchy problem for the equation (\ref{heatEq}) in the classical sense in terms of Duhamel's principle. Prior to addressing the solution, it is necessary to introduce some preliminaries.
Before proceeding with these preliminaries, note that we set $$\operatorname{W}_{t,t}f=f,\quad \forall f\in \mathcal{S}(\mathbb{R}^n), ~ \forall t\in [0,\infty).$$
\begin{remark}\label{Leibniz}
Assume that $f\in C^\infty(\mathbb{R}_{+},\mathcal{S}(\mathbb{R}^n))$. Then,
\begin{align}
	\partial_{t}^\gamma\left(\int\limits_{0}^{t}\operatorname{W}_{s,t}f(s,x)ds\right)&=	\partial_{t}^\gamma\left(\int\limits_{0}^{t}\mathcal{W}(\cdot;s,t)\ast f(s)(x)ds\right)=\gamma\partial_{t}^{\gamma-1}W_{t,t}f(t,x)+\int\limits_{0}^{t}\partial_{t}^{\gamma}\left(\operatorname{W}_{s,t}f(s,x)\right)ds\notag\\
	&=\gamma\partial_{t}^{\gamma-1}f(t,x)+\int\limits_{0}^{t}\partial_{t}^{\gamma}\operatorname{W}_{s,t}f(s,x)ds,\\
	&\quad\forall x\in \mathbb{R}^n,~~\forall t\in \mathbb{R}_{+}, ~~\forall \gamma\in \mathbb{N}\notag.
\end{align}
The differentiation under the integral sign is justified because $f\in C^\infty(\mathbb{R}_+,\mathcal{S}(\mathbb{R}^n))$ and, by part $1$ of Proposition \ref{pr2}, for each $t\in\mathbb{R}_+$ the map  
$$
s\longmapsto \operatorname{W}_{s,t}f(s),
$$ 
is continuous from $[0,\infty)$ into $\mathcal{S}(\mathbb{R}^n)$.
Moreover, part $2$ of  Proposition \ref{pr2} implies that for every $k\in\mathbb{N}$ the map  
$$
(s,t)\longmapsto \partial_t^k\bigl(\operatorname{W}_{s,t}f(s)\bigr),
$$  
is continuous on $\Delta$ with values in $\mathcal{S}(\mathbb{R}^n)$.
Consequently, for any fixed $t>0$, the integral  
$$
\int_0^t \operatorname{W}_{s,t}f(s)\,ds,
$$  
is well--defined as an element of $\mathcal{S}(\mathbb{R}^n)$, and differentiation with respect to $t$ under the integral sign is justified.
\end{remark}

\begin{lemma}\label{avoidingstot}
\begin{itemize}
\item[$1.$]
	Let $f \in C^{\infty}(\mathbb{R}_{+}, \mathcal{S}(\mathbb{R}^{n}))$, and fix $\alpha, \beta \in \mathbb{N}_0^n$. Then for all $(s,t) \in \Delta$ and $\gamma \in \mathbb{N}_0$, we have
\begin{align}\label{lemma21}
	&\sup_{x \in \mathbb{R}^n} \Big| x^\alpha \big( \partial_t^\gamma \mathcal{W}(\cdot; s, t) \ast \partial_x^\beta f(s, \cdot) \big)(x) \Big|\notag\\
	&\lesssim_{\alpha,\beta,\gamma} \sum_{m \le \alpha} \sum_{\zeta \le \alpha-m} 
	\sum_{\mu'=1}^{|m|} \sum_{\mu=1}^\gamma 
	\left( \int_0^t \|a(\tau)\|_{\mathrm{op}} \, d\tau \right)^{\mu'} 
	\Theta_{\mu,\gamma}(t) \,
	I_{m,\zeta,\mu',\mu,\beta}(f(s)),
\end{align}
where
$$
\Theta_{\mu,\gamma}(t)\doteq  \left(\max_{1 \le q \le \gamma} \|\partial_t^{q} A(t)\|_{\mathrm{op}}\right)^{\mu}, \quad	I_{m,\zeta,\mu',\mu,\beta}(f(s)) \doteq \int_{\mathbb{R}^n} (1+|\xi|)^{\mu'+\mu+|\beta|-|\zeta|} \big| \partial_\xi^{\alpha-m-\zeta} \widehat{f}(s,\xi) \big| \, d\xi.
$$
Moreover, each $I_{m,\zeta,\mu',\mu,\beta}(f(s))$ is bounded by a finite sum of Schwartz seminorms of $f(s)$. Explicitly, for integers $N,N'$ satisfying
$$
N > \mu'+\mu+|\beta|-|\zeta| + n, \quad
N' > |\alpha-m-\zeta|+	N+ n,
$$
we have
$$
I_{m,\zeta,\mu',\mu,\beta}(f(s)) \lesssim_{\alpha,m,\zeta,\mu',\mu,\beta,N,N',n} \sum_{|k'|\le N'}\sum_{|b|\le N}\|f(s)\|_{k',b}.
$$

In particular, for $\gamma = 0$, the estimate simplifies to
\begin{align}\label{lemma22}
	\sup_{x \in \mathbb{R}^n} 
	\Big| x^\alpha \big( \mathcal{W}(\cdot; s, t) \ast \partial_x^\beta f(s, \cdot) \big)(x) \Big| 
	&\lesssim_{\alpha,\beta} \sum_{m\le \alpha} \sum_{\zeta\le \alpha-m} \sum_{\mu'=1}^{|m|} 
	\left(\int_{0}^{t}\| a(\tau)\|_{\mathrm{op}}\,d\tau\right)^{\mu'} I_{m,\zeta,\mu',\beta}(f(s)),
\end{align}
with
\[
I_{m,\zeta,\mu',\beta}(f(s))
\doteq \int_{\mathbb{R}^n} (1 + |\xi|)^{\mu' + |\beta| - |\zeta|} 
\big| \partial_\xi^{\alpha - m - \zeta} \hat{f}(s, \xi) \big| \, d\xi,
\]
and the corresponding bound
\[
I_{m,\zeta,\mu',\beta}(f(s))
\lesssim_{N,N',n} \sum_{|k'|\le N'}\sum_{|b|\le N}\|f(s)\|_{k',b}.\]
\item [$2$] Let $\mathcal{B}(\mathcal{S}(\mathbb{R}^n))$ denote the space of continuous linear operators on $\mathcal{S}(\mathbb{R}^n)$, and let $C \subset \overline{\Delta}$ be compact. Then, the family
	$
	\{\operatorname{W}_{s,t}\}_{(s,t)\in C} \subset \mathcal{B}(\mathcal{S}(\mathbb{R}^n))$ is equicontinuous.
	
	\item [$3$] Let $1\leq q \leq \infty$, and let $\mathcal{B}(W^{q,\infty}(\mathbb{R}^n))$ denote the space of continuous linear operators on $W^{q,\infty}(\mathbb{R}^n)$, and let $C \subset \overline{\Delta}$ be compact. Then, the family 
	$
	\{\operatorname{W}_{s,t}\}_{(s,t)\in C} \subset \mathcal{B}(W^{q,\infty}(\mathbb{R}^n))
	$ 
	is equicontinuous.
	\end{itemize}
\end{lemma}
\begin{proof}
We prove the three parts separately. To prove part 1, fix $\alpha, \beta \in \mathbb{N}_0^n$, $\gamma \in \mathbb{N}_0$, and $(s,t) \in \Delta$. Let
$$
F_{s,t}(x) \doteq x^\alpha \big( \partial_t^\gamma \mathcal{W}(\cdot; s, t) \ast \partial_x^\beta f(s, \cdot) \big)(x), \quad \forall x \in \mathbb{R}^n.
$$
Using the inequality
\begin{equation}\label{inequalityestimate}\|F_{s,t}\|_\infty \leq \|\widehat{F_{s,t}}\|_1,\end{equation} we obtain
\begin{align}\label{fst}
\sup_{x\in\mathbb{R}^n}|F_{s,t}(x)| &\leq \int_{\mathbb{R}^n} \big|\widehat{F_{s,t}}(\xi)\big|\,d\xi\notag \\
&= (2\pi)^{-|\alpha|}\int_{\mathbb{R}^n}\Big| \partial_\xi^\alpha \big[ \partial_t^\gamma \widehat{\mathcal{W}}(\xi;s,t)  \widehat{\partial_{\xi}^{\beta}f(s,\xi)} \big] \Big|d\xi
\notag\\
&\leq (2\pi)^{|\beta|-|\alpha|} \int_{\mathbb{R}^n} \Big| \partial_\xi^\alpha \big[ \partial_t^\gamma \widehat{\mathcal{W}}(\xi;s,t) \, \xi^\beta \widehat{f}(s,\xi) \big] \Big|\,d\xi.
\end{align}
In the next step we use the Leibniz rule two times to obtain for all $\xi\in\mathbb{R}^{n}$ and $\zeta\in\mathbb{N}^{n}_{0}$ such that $\zeta<\beta$
\begin{align}
\partial_\xi^\alpha \big[ \partial_t^\gamma \widehat{\mathcal{W}}(\xi;s,t) \, \xi^\beta \widehat{f}(s) \big]
&= \sum_{m \le \alpha} \binom{\alpha}{m} \big(\partial_\xi^m \partial_t^\gamma \widehat{\mathcal{W}}(\xi;s,t)\big) 
\partial_\xi^{\alpha-m} \big[ \xi^\beta \widehat{f}(s, \xi) \big]\notag\\
&= \sum_{m \le \alpha} \sum_{\zeta \le \alpha-m} \binom{\alpha}{m} \binom{\alpha-m}{\zeta} \big(\partial_\xi^m \partial_t^\gamma \widehat{\mathcal{W}}(\xi;s,t)\big)
\big( \partial_\xi^\zeta \xi^\beta \big) \,
\partial_\xi^{\alpha-m-\zeta} \widehat{f}(s,\xi),
\end{align}
with $|\partial_\xi^\zeta \xi^\beta| \lesssim_{\beta,\zeta} (1+|\xi|)^{|\beta|-|\zeta|}$.

We now analyze $\partial_\xi^m \partial_t^\gamma \widehat{\mathcal{W}}(\xi;s,t)$ for all $\xi \in \mathbb{R}^n$, $(s,t)\in\Delta$, $\gamma\in\mathbb{N}_0$, and $m\in\mathbb{N}^{n}_{0}$. Using Fa\`a di Bruno's formula for the time derivative, then applying the Leibniz rule, one has
\begin{align}\label{eqfaleib5}
&\left|\partial_\xi^m \partial_t^\gamma \widehat{\mathcal{W}}(\xi;s,t)\right|= \notag\\
&\left|\sum_{m' \le m} \binom{m}{m'} 
\big( \partial_\xi^{m'} e^{-\langle (A(t)-A(s))\xi, \xi \rangle} \big) 
\partial_\xi^{m-m'}\left(\sum_{\mu=1}^\gamma (-1)^\mu 
\sum_{\substack{p_1+\cdots+p_\mu = \gamma \\ p_j \ge 1}} 
\frac{\gamma!}{p_1!\cdots p_\mu!} \prod_{j=1}^\mu \partial_t^{p_j} \langle (A(t)-A(s))\xi, \xi\rangle\right)\right|\notag\\
&\lesssim_{m} \sum_{m'\leq m}\left|\partial_\xi^{m'} e^{-\langle (A(t)-A(s))\xi, \xi \rangle} \right|\,\left|	\partial_\xi^{m-m'}\left(\sum_{\mu=1}^\gamma (-1)^\mu 
\sum_{\substack{p_1+\cdots+p_\mu = \gamma \\ p_j >0}} 
\frac{\gamma!}{p_1!\cdots p_\mu!} \prod_{j=1}^\mu \partial_t^{p_j} \langle (A(t)-A(s))\xi, \xi\rangle\right)\right|.
\end{align}

Now compute each factor in \eqref{eqfaleib5} separately. 
First, by Fa\`a di Bruno's formula applied to 
$\partial_\xi^{m'} e^{-\langle (A(t)-A(s))\xi, \xi \rangle}$, one has for all $\xi \in \mathbb{R}^n$ and $m' \in \mathbb{N}_0^n$
\begin{align}\label{firstterm}
&\left|\partial_\xi^{m'} e^{-\langle (A(t)-A(s))\xi, \xi \rangle}\right|= \left|e^{-\langle (A(t)-A(s))\xi, \xi \rangle} \sum_{\mu'=1}^{|m'|} (-1)^{\mu'} 
\sum_{\substack{q_1+\cdots+q_{\mu'} = m' \\ q_j \ne 0}} 
\frac{m'!}{q_1!\cdots q_{\mu'}!} \prod_{j=1}^{\mu'} \partial_\xi^{q_j} \langle (A(t)-A(s))\xi, \xi \rangle\right|\notag\\
&\leq m'! \sum_{\mu'=1}^{|m'|} 	\sum_{\substack{q_1+\cdots+q_{\mu'} = m' \\ q_j \ne 0}}\prod_{j=1}^{\mu'} \left|\partial_\xi^{q_j} \langle (A(t)-A(s))\xi, \xi \rangle \right|\leq  m'! \sum_{\mu'=1}^{|m'|} 	\sum_{\substack{q_1+\cdots+q_{\mu'} = m' \\ q_j \ne 0}}\prod_{j=1}^{\mu'} 2\left(\int_0^t \|a(\tau)\|_{\mathrm{op}} d\tau\right) (1+|\xi|)\notag\\
&\lesssim_{m'} \sum_{\mu'=1}^{|m'|} \left(\int_0^t \|a(\tau)\|_{\mathrm{op}} d\tau\right)^{\mu'} (1+|\xi|)^{\mu'},
\end{align}
where we used that for any multi-index $q \in \mathbb{N}^n$,
\begin{align}\label{estimatederivative}
|\partial_\xi^{q} \langle (A(t)-A(s))\xi,\xi\rangle| &\leq 2\|(A(t)-A(s))\|_{\mathrm{op}}(1+|\xi|)\leq
2	\left(	\int_s^t \|a(\tau)\|_{\mathrm{op}} d\tau\right) (1+|\xi|)\notag\\&\leq 2\left(\int_0^t \|a(\tau)\|_{\mathrm{op}} d\tau\right) (1+|\xi|),
\end{align}
which follows by reasoning analogous to \eqref{derivative1}, \eqref{derivative2}, and \eqref{derivative33}. Next, applying the Leibniz rule to the second term of \eqref{eqfaleib5} yields, for all $\xi \in \mathbb{R}^n$ and $m,m' \in \mathbb{N}_0^n$ with $m \ge m'$, 
\begin{align}\label{derivativeestimate3}
&\partial_\xi^{m-m'}\left(\sum_{\mu=1}^\gamma (-1)^\mu 
\sum_{\substack{p_1+\cdots+p_\mu = \gamma \\ p_j \neq 0}} 
\frac{\gamma!}{p_1!\cdots p_\mu!} \prod_{j=1}^\mu \partial_t^{p_j} \langle (A(t)-A(s))\xi, \xi\rangle\right)\notag\\
&= \sum_{\mu=1}^\gamma (-1)^\mu 
\sum_{\substack{p_1+\cdots+p_\mu = \gamma \\ p_j \neq 0}} 
\frac{\gamma!}{p_1!\cdots p_\mu!} \sum_{\substack{r_1+\cdots+r_\mu = m-m' \\ r_j \ge 0}} 
\frac{(m-m')!}{r_1!\cdots r_\mu!} \prod_{j=1}^\mu \partial_\xi^{r_j}  \langle \partial_t^{p_j}A(t)\xi, \xi\rangle\displaybreak\notag\\
&\leq \gamma!(m-m')! \sum_{\mu=1}^\gamma 
\sum_{\substack{p_1+\cdots+p_\mu = \gamma \\ p_j \neq 0}} 
\sum_{\substack{r_1+\cdots+r_\mu = m-m' \\ r_j \ge 0}} 
\prod_{j=1}^\mu \left|\partial_\xi^{r_j}  \langle \partial_t^{p_j}A(t)\xi, \xi\rangle\right|\notag\\
&\leq \gamma!(m-m')! \sum_{\mu=1}^\gamma 
\sum_{\substack{p_1+\cdots+p_\mu = \gamma \\ p_j \neq 0}} 
\sum_{\substack{r_1+\cdots+r_\mu = m-m' \\ r_j \ge 0}} 
\prod_{j=1}^\mu 2\left\|\partial_t^{p_j} A(t)\right\|_{\mathrm{op}} (1+|\xi|)\notag\\
&\lesssim_{\gamma,m,m'} \sum_{\mu=1}^\gamma (1+|\xi|)^{\mu}\left(\max_{0< q<\gamma} \left\|\partial_t^{q} A(t)\right\|_{\mathrm{op}}\right)^{\mu},	 
\end{align}
where for any indices $r\in\mathbb{N}^{n}_{0}$ and $p\in\mathbb{N}$, we used that 
\begin{align}\label{estimateupperboundderivative}
\left|\partial_\xi^{r}  \langle \partial_t^{p}A(t)\xi, \xi \rangle\right|\leq 2\left\|\partial_t^{p} A(t)\right\|_{\mathrm{op}} (1+|\xi|),
\end{align}
which follows by reasoning analogous to \eqref{derivative1}, \eqref{derivative2}, and \eqref{derivative33}.
Therefore, combining  \eqref{firstterm}, \eqref{derivativeestimate3}, we obtain
\begin{align*}
|\partial_\xi^m \partial_t^\gamma \widehat{\mathcal{W}}(\xi;s,t)|&\lesssim_{m,\gamma}  \sum_{m'\leq m}\sum_{\mu'=1}^{|m'|} \sum_{\mu=1}^\gamma 
\left( \int_0^t \|a(\tau)\|_{\mathrm{op}}\,d\tau \right)^{\mu'}  
\left(\max_{0<q<\gamma} \|\partial_t^{q} A(t)\|_{\mathrm{op}}\right)^{\mu} 
(1+|\xi|)^{\mu'+\mu} \\
&\lesssim_{m} \sum_{\mu'=1}^{|m|} \sum_{\mu=1}^\gamma 
\left( \int_0^t \|a(\tau)\|_{\mathrm{op}}\,d\tau \right)^{\mu'} 
\left(\max_{0<q<\gamma} \|\partial_t^{q} A(t)\|_{\mathrm{op}}\right)^{\mu} (1+|\xi|)^{\mu'+\mu}\notag\\
&\doteq \sum_{\mu'=1}^{|m|} \sum_{\mu=1}^\gamma 
\left( \int_0^t \|a(\tau)\|_{\mathrm{op}}\,d\tau \right)^{\mu'} 
\Theta_{\mu,\gamma}(t)	 (1+|\xi|)^{\mu'+\mu},~~~\forall \xi\in\mathbb{R}^{n},~~\forall m\in\mathbb{N}^{n}_{0}.
\end{align*}
Going back to \eqref{fst}, one can obtain
\begin{align*}
&\sup_{x\in\mathbb{R}^n}|F_{s,t}(x)|\notag\\ &\lesssim_{\alpha,\beta,\gamma} \sum_{m \le \alpha} \sum_{\zeta \le \alpha-m} 
\sum_{\mu'=1}^{|m|} \sum_{\mu=1}^\gamma 
\left( \int_0^t \|a(\tau)\|_{\mathrm{op}}\,d\tau \right)^{\mu'}\Theta_{\mu,\gamma}(t)
\int_{\mathbb{R}^n} (1+|\xi|)^{\mu'+\mu+|\beta|-|\zeta|} 
\big| \partial_\xi^{\alpha-m-\zeta} \widehat{f}(s,\xi) \big| \, d\xi\notag\\
&\doteq \sum_{m \le \alpha} \sum_{\zeta \le \alpha-m} 
\sum_{\mu'=1}^{|m|} \sum_{\mu=1}^\gamma 
\left( \int_0^t \|a(\tau)\|_{\mathrm{op}}\,d\tau \right)^{\mu'}\Theta_{\mu,\gamma}(t)I_{m,\zeta,\mu',\mu,\beta}(f(s)).
\end{align*}
We now bound the term $I_{m,\zeta,\mu',\mu,\beta}(f(s)),~~\forall  m,\zeta, \beta\in\mathbb{N}^{n}_{0},~~\forall \mu,\mu'\in\mathbb{N}$. Since $\widehat{f}(s) \in \mathcal{S}(\mathbb{R}^n)$, choose integer $N$ such that
$
N > \mu'+\mu+|\beta|-|\zeta| + n
$.
Then, using the integrability of $(1+|\xi|)^{\mu'+\mu+|\beta|-|\zeta|-N}$ and the fact that
$$
(1+|\xi|)^N \lesssim_{N,n} \sum_{|k| \le N} |\xi^k|,~~\forall \xi\in\mathbb{R}^{n},
$$
we compute
\begin{align*}
I_{m,\zeta,\mu',\mu,\beta}(f(s)) 
&= \int_{\mathbb{R}^n} (1+|\xi|)^{\mu'+\mu+|\beta|-|\zeta|-N} 
(1+|\xi|)^N \big| \partial_\xi^{\alpha-m-\zeta} \widehat{f}(s,\xi) \big| \, d\xi \\
&\lesssim_N \sup_{\xi \in \mathbb{R}^n} (1+|\xi|)^N 
\big| \partial_\xi^{\alpha-m-\zeta} \widehat{f}(s,\xi) \big|\int_{\mathbb{R}^n} (1+|\xi|)^{\mu'+\mu+|\beta|-|\zeta|-N} d\xi \\
&\lesssim_{N,n,\mu,\mu',\beta,\zeta} \sum_{|k| \le N} \sup_{\xi \in \mathbb{R}^n} |\xi^k| 
\big| \partial_\xi^{\alpha-m-\zeta} \widehat{f}(s,\xi) \big|,\\
&~~~\forall  m,\zeta\in\mathbb{N}^{n}_{0},~~\forall \mu,\mu'\in\mathbb{N},~~\text{with}~\alpha-m-\zeta\geq 0.
\end{align*}
Using the relation
\[
\partial_\xi^{\alpha-m-\zeta} \widehat{f}(s,\cdot)(\xi)
=
\widehat{(-2\pi i x)^{\alpha-m-\zeta} f(s,\cdot)}(\xi),~~\forall \zeta,m\in\mathbb{N}^{n}_{0}  ~~\text{with}~~ \alpha-m-\zeta\geq 0,
\]
together with inequality~\eqref{inequalityestimate}, the Leibniz rule, and the formula for the derivative of a monomial, we proceed as follows:
\begin{align*}
&I_{m,\zeta,\mu',\mu,\beta}(f(s))\lesssim_{\alpha,m,\zeta,N,n,\mu,\mu',\beta} \sum_{|k|\leq N}\sup_{\xi\in\mathbb{R}^n}\left|\widehat{\partial_{x}^k((\cdot)^{\alpha-m-\zeta}f(s,\cdot))(\xi)}\right|\notag\\
&\leq\sum_{|k|\leq N}\int\limits_{\mathbb{R}^{n}}\left| \partial_{x}^k(x^{\alpha-m-\zeta}f(s,\cdot))(x)\right|dx\leq\sum_{|k|\leq N}\sum_{\ell\leq k}\binom{k}{\ell}\int\limits_{\mathbb{R}^{n}}|\partial^{\ell}_{x}(x^{\alpha-m-\zeta})\partial^{k-\ell}_{x}f(s,x)|dx\notag\\
&\lesssim_{N} \sum_{|k|\leq N}\sum_{\ell\leq k}\binom{k}{\ell}\binom{\alpha-m-\zeta}{\ell}\int\limits_{\mathbb{R}^{n}}|x^{\alpha-m-\zeta-\ell}\partial^{k-\ell}_{x}f(s,x)|dx\notag\\
&\lesssim_{\alpha,m,\zeta,N} \sum_{|k|\leq N}\sum_{\ell\leq k}\int\limits_{\mathbb{R}^{n}}(1+|x|)^{|\alpha-m-\zeta-\ell|}|\partial^{k-\ell}_{x}f(s,x)|dx,\\
&\forall  m,\zeta \in\mathbb{N}^{n}_{0},~~\forall \mu,\mu'\in\mathbb{N}.
\end{align*}
Now for $N' >| \alpha-m-\zeta|+N+ n$ and use similar reasoning, one has
\begin{align*}
&I_{m,\zeta,\mu',\mu,\beta}(f(s))\lesssim_{	m,\zeta,\mu',\mu,\beta,N,\alpha,n} \sum_{|k|\leq N}\sum_{\ell\leq k}\int\limits_{\mathbb{R}^{n}}(1+|x|)^{|\alpha-m-\zeta-\ell|-N'}(1+|x|)^{N'}|\partial^{k-\ell}_{x}f(s,x)|dx\notag\\
&\lesssim_{\alpha,m,\zeta,N,N',n}\sum_{|k|\leq N}\sum_{\ell\leq k}\sum_{|k'|\leq N'}\sup_{\xi\in\mathbb{R}^n}\left|x^{k'}\partial_x^{k-\ell} f(s,x)\right|= \sum_{|k|\leq N}\sum_{\ell\leq k}\sum_{|k'|\leq N'}\|f(s)\|_{k',k-\ell}
\notag\\&\lesssim_{N,N'}
\sum_{|k'|\le N'}\sum_{|b|\le N}\|f(s)\|_{k',b},\quad\forall m,\zeta \in\mathbb{N}^{n}_{0},~~\forall \mu,\mu'\in\mathbb{N}. 
\end{align*}
The proof for $\gamma=0$ is entirely analogous.	This completes the proof of part 1.

To prove the equicontinuity of the family $\{\operatorname{W}_{s,t}\}_{(s,t)\in C}$ in Part~2, it suffices to show that for every Schwartz seminorm $\|\cdot\|_{\alpha,\beta}$ on $\mathcal{S}(\mathbb{R}^n)$ with $\alpha,\beta\in\mathbb{N}_0^n$, there exist a finite set $\{(\alpha_i,\beta_i)\}_{i=1}^d\subset\mathbb{N}_0^{2n}$, for some $d\in\mathbb{N}$, and a constant $M>0$ such that

\[
\| \operatorname{W}_{s,t} f \|_{\alpha,\beta} 
\le M \sum_{i=1}^{d} \| f \|_{\alpha_i,\beta_i}, 
\qquad \forall f \in \mathcal{S}(\mathbb{R}^n), \ \forall (s,t) \in C.
\]

Fix $(s,t)\in C$ and let $f\in\mathcal{S}(\mathbb{R}^{n})$. We apply the representation corresponding to the case $\gamma=0$ in part $1$ to estimate
\begin{align}
	&\sup_{x\in\mathbb{R}^{n}} |x^{\alpha}\partial^{\beta}_{x} \operatorname{W}_{s,t} f(x)|
	= \sup_{x\in\mathbb{R}^{n}} \big| x^\alpha \partial_x^\beta (\mathcal{W}(\cdot;s,t) \ast f)(x) \big|
	= \sup_{x\in\mathbb{R}^{n}} \big| x^\alpha (\mathcal{W}(\cdot;s,t) \ast \partial_x^\beta f)(x) \big| \notag\\
	&\lesssim_{\alpha,\beta, n}
	\sum_{m\le \alpha}
	\sum_{\zeta\le \alpha-m}
	\sum_{\mu'=1}^{|m|}
	\left(\int\limits_{0}^{t}\| a(\tau)\|_{op}\,d\tau\right)^{\mu'}\, 
	\sum_{|k'|\leq N'} \sum_{|b|\leq N}
	\|f\|_{k', b},
\end{align}
where $\|f\|_{k',b}$ denotes a Schwartz seminorm of $f$, and the parameters $N$ and $N'$ satisfy
$$
N > \mu' + |\beta| - |\zeta| + n, \qquad N' >| \alpha - m - \zeta|+N+n,
$$
where $\alpha, \beta, m, \zeta\in\mathbb{N}^{n}_{0}$, $\mu'\in\mathbb{N}$.
Here, since $C$ is compact, its projection onto the $t$-axis,
$$
\pi_t(C) \doteq \{ t \in \mathbb{R}_+ : \exists s \geq 0~ \text{ such that } (s,t)\in C \},
$$
is also compact. Taking the supremum over $t \in \pi_t(C)$ in all $t$-dependent factors gives an upper bound independent of $(s,t) \in C$, yielding
\begin{align}
	\sup_{x\in\mathbb{R}^{n}} |x^\alpha\partial_x^{\beta} \operatorname{W}_{s,t} f(x)| 
	&\lesssim_{\alpha,\beta}
	\sum_{m\le \alpha}
	\sum_{\zeta\le \alpha-m}
	\sum_{\mu'=1}^{|m|}
	\left(\int\limits_{0}^{\max{\pi_t(C)}}\| a(\tau)\|_{op}\,d\tau\right)^{\mu'}\, 
	\sum_{|b|\leq N}\sum_{|k'|\leq N'} 
	\|f\|_{k', b}.
\end{align}
Hence, the family $\{\operatorname{W}_{s,t}\}_{(s,t)\in C}$ is equicontinuous on $C$, which completes the proof of part $2$.

To prove part~$3$, it suffices to establish the equicontinuity of the family
\[
\{ \operatorname{W}_{s,t} \}_{(s,t)\in C} \subset B(W^{q,\infty}(\mathbb{R}^n)),
\]
which can be achieved by showing the existence of a constant \(M>0\) such that
$$
\|\partial^{\beta}_{x}(\operatorname{W}_{s,t}f) \|_{q} \le M \|\partial^{\beta}_{x} f\|_{q}, 
\quad \forall\beta\in\mathbb{N}_0^n, ~~\forall f \in W^{q,\infty}(\mathbb{R}^n), \ \forall (s,t) \in C.
$$
Similar to the argument in part~$2$, fix $(s,t)\in C$ and let $f\in W^{q,\infty}(\mathbb{R}^{n})$. Using Young's inequality and part~$1$ of Proposition~\ref{pr1}, we have
\begin{align}\label{eqctsWq}
	\|\partial^{\beta}_{x} (\operatorname{W}_{s,t} f)\|_{q} 
	&= \|\partial^{\beta}_{x}(\mathcal{W}(\cdot;s,t) \ast f)\|_{q} 
	= \|\mathcal{W}(\cdot;s,t) \ast \partial^{\beta}_{x} f\|_{q} \notag\\
	&\le \|\mathcal{W}(\cdot;s,t)\|_1 \, \|\partial^{\beta}_{x} f\|_{q} 
	\le \|\partial^{\beta}_{x} f\|_{q} < \infty, \quad \forall\beta\in\mathbb{N}_0^n.
\end{align} 
Since the upper bound in \eqref{eqctsWq} is independent of $(s,t)\in C$, the equicontinuity of 
\(\{\operatorname{W}_{s,t}\}_{(s,t)\in C}\) follows. This completes the proof of part $3$.
\end{proof}

The following remark introduces the necessary notation related to seminorms in locally convex topological spaces $C^{\infty}(\mathbb{R}_+, \mathcal{S}(\mathbb{R}^n))$, $C([0,+\infty),\mathcal{S}(\mathbb{R}^n))$, as well as $C^\infty(\mathbb{R}_+,W^{q,\infty}(\mathbb{R}^n))$ and $C([0,+\infty),W^{q,\infty}(\mathbb{R}^n))$ which will be used in the sequel.
\begin{remark}\label{notation2}
	Let $h\in C^{\infty}(\mathbb{R}_+, \mathcal{S}(\mathbb{R}^n))$ and $K\subset\mathbb{R}_+$ be a compact set. Then, 
	for every $\alpha, \beta\in\mathbb{N}_0^n$, $\gamma\in\mathbb{N}_0$, we denote
	\begin{equation}
		\left\| h \right\|_{\alpha,\beta,\gamma, K} 
		\doteq \sup\limits_{t \in K} \sup\limits_{x \in \mathbb{R}^n} \left| x^\alpha \partial_x^\beta \partial_t^\gamma  h (t, x) \right|.
	\end{equation}
	Moreover, let $h\in C([0,+\infty),\mathcal{S}(\mathbb{R}^n))$ and $b\in\mathbb{N}$. Then, we denote for every $\alpha,\beta \in \mathbb{N}_0^n$
	\begin{equation}
		\|h\|_{\alpha,\beta,b}\doteq \sup_{t \in [0,b]}\|h(t)\|_{\alpha,\beta}.
	\end{equation}
	In addition, let $q\in [1,\infty)$. Let $h\in C^\infty(\mathbb{R}_+,W^{q,\infty}(\mathbb{R}^n))$ and $K\subset\mathbb{R}_+$ be a compact set. Then, 
	for every $j\in\mathbb{N}_0$, $\ell\in\mathbb{N}_0^n$, and  we denote
	$$
	\|h\|_{q,j,\ell,K}\doteq\sup_{t\in K}\left\|\partial^{j}_{t}\partial_{x}^{\ell} h (t)\right\|_q.
	$$	Additionally,
	let $h\in C([0,+\infty),L^q(\mathbb{R}^n))$ and $b\in\mathbb{N}$, we denote
	$$
	\|h\|_{q,b}\doteq \sup_{t\in[0,b]}\|h(t)\|_q.
	$$
	Finally, let $h\in C([0,\infty), W^{q,\infty}(\mathbb{R}^{n}))$ and $b\in\mathbb{N}$. Then, for every $\ell\in\mathbb{N}^{n}_{0}$ we denote
	$$	\|h\|_{q,\ell,b}\doteq \sup_{t\in[0,b]}\|\partial^{\ell}_{x}h(t)\|_q.$$
\end{remark}

Having established the necessary framework, we are now ready to proceed with the key proposition.

\begin{proposition}\label{pr3}
Let $u_0\in \mathcal{S}
(\mathbb{R}^n)$ and $f\in  C^\infty(\mathbb{R}_+, \mathcal{S}(\mathbb{R}^n))\,\cap\, C([0,\infty), \mathcal{S}(\mathbb{R}^n))$. Then the function $u:\mathbb{R}_+\times\mathbb{R}^n\to\mathbb{C}$ defined by
\begin{equation}\label{solution}
	u(t,x)\doteq\operatorname{W}_{0,t}u_0(x)+\int\limits_0^t\operatorname{W}_{s,t}f(s,x)ds,\quad \forall x\in\mathbb{R}^n,~~ \forall t\in\mathbb{R}_+,
\end{equation}
satisfies
\begin{equation}\label{tensor}
	u\in C^\infty(\mathbb{R}_+, \mathcal{S}(\mathbb{R}^n))\,\cap\, C([0,\infty), \mathcal{S}(\mathbb{R}^{n})),
\end{equation}
and the Cauchy problem
\begin{equation}\label{cauchy}
\begin{cases}
	\partial_t u(t,x)-\mathcal{L}_{t}u(t,x)=f(t,x),\quad \forall x\in\mathbb{R}^n,~~\forall t\in\mathbb{R}_+,\\
	\lim\limits_{t\to0+}u(t,\cdot)=u_0.
\end{cases}
\end{equation}
Moreover, if $u_0\ge0$ and $f\ge0$ then $u\ge0$.
\end{proposition}

\begin{proof} 
To show \eqref{tensor}, it suffices to prove the following:
\begin{align}\label{seminorm1}
	&\operatorname{W}_{0,\cdot}u_{0} \in C^{\infty}(\mathbb{R}_{+},\mathcal{S} (\mathbb{R}^n))\cap C([0,\infty),\mathcal{S}(\mathbb{R}^{n})),
\end{align}
i.e.
\begin{align}\label{seminorm2}
&\forall \alpha,\beta\in \mathbb{N}^{n}_{0},~~\forall \gamma\in \mathbb{N}_{0}, ~~\forall~\text{compact}~K\subset \mathbb{R}_{+}:\notag\\	&\|\operatorname{W}_{0,\cdot}u_{0}\|_{\alpha,\beta,\gamma,K}=\sup_{x\in\mathbb{R}^n}\sup_{t\in K\subset \mathbb{R}_{+}}\left|x^{\alpha}\partial^{\beta}_{x}\partial^{\gamma}_{t}\operatorname{W}_{0,t}u_{0}(x)\right|<\infty,\notag\\
\end{align}
\begin{align}\label{cntsatzero}
	& \forall b\in \mathbb{N},~~\forall\alpha,\beta\in \mathbb{N}^{n}_{0}:\notag\\
&\|\operatorname{W}_{0,\cdot}u_{0}\|_{\alpha,\beta,b}=\sup_{t\in[0,b]}\sup_{x\in\mathbb{R}^n}\left|x^\alpha\partial_x^\beta\left(\mathrm{W}_{0,t} u_{0}\right) (x) \right|<\infty,
\end{align}
together with
\begin{align}\label{seminorm3}
	&\int\limits_0^\cdot\operatorname{W}_{s,\cdot}f(s)ds\in C^{\infty}(\mathbb{R}_{+},\mathcal{S}( \mathbb{R}^n))\cap C([0,\infty),\mathcal{S}(\mathbb{R}^{n})),
\end{align}
i.e.
\begin{align}\label{seminorm4}
&\forall \alpha,\beta\in \mathbb{N}^{n}_{0},~~\forall \gamma\in \mathbb{N}_{0},~~\forall~ \text{compact}~K\subset \mathbb{R}_{+}:\notag\\	&\left\|\int\limits_0^\cdot\operatorname{W}_{s,\cdot}f(s)ds\right\|_{\alpha,\beta,\gamma,K}=\sup_{x\in\mathbb{R}^n}\sup_{t\in K\subset \mathbb{R}_{+}}\left|x^{\alpha}\partial^{\beta}_{x}\partial^{\gamma}_{t}\int\limits_0^t\operatorname{W}_{s,t}f(s,x)ds\right|<\infty,
\end{align}
and finally,
\begin{align}\label{cntsatzerointegral}
	&\forall b\in \mathbb{N},~~\forall\alpha,\beta\in \mathbb{N}^{n}_{0}:\notag\\
	&\left\|\int\limits_0^\cdot\operatorname{W}_{s,\cdot}f(s)ds\right\|_{\alpha,\beta,b}=\sup_{t\in [0,b]}\sup_{x\in\mathbb{R}^n}\left| x^\alpha\partial_x^\beta\int\limits_0^t\operatorname{W}_{s,t}f(s,x)ds\right|<\infty. 
\end{align}
To prove \eqref{seminorm2}, observe that since $\mathcal{W}\in C^\infty(\Delta, \mathcal{S}(\mathbb{R}^n))$, the mapping \(t \mapsto \mathcal{W}( \cdot; 0, t) \in \mathcal{S}(\mathbb{R}^n)\) belongs to \(C^\infty(\mathbb{R}_+, \mathcal{S}(\mathbb{R}^n))\). Therefore, for every compact $K\subset \mathbb{R}_+$, one has
\begin{align*}
	\sup_{x\in\mathbb{R}^n}\sup_{t\in K\subset \mathbb{R}_{+}}\left|x^{\alpha}\partial^{\beta}_{x}\partial^{\gamma}_{t}\operatorname{W}_{0,t}u_{0}(x)\right|&=\sup_{x\in\mathbb{R}^n}\sup_{t\in K\subset \mathbb{R}_{+}}\left|x^{\alpha}\partial^{\beta}_{x}\partial^{\gamma}_{t}((\mathcal{W}(\cdot;0,t)\ast u_{0})(x))\right|\\
	&=\sup_{x\in\mathbb{R}^n}\sup_{t\in K\subset \mathbb{R}_{+}}\left|x^{\alpha}\partial^{\beta}_{x}\big(\partial^{\gamma}_{t}\mathcal{W}(\cdot;0,t)\ast u_{0}\big)\right|.
\end{align*}
Moreover, the fact that $\mathcal{W}\in C^\infty(\Delta, \mathcal{S}(\mathbb{R}^n))$ implies that the mapping $x\mapsto \partial^{\gamma}_{t}\mathcal{W}(\cdot,0,t)\in \mathcal{S}(\mathbb{R}^n)$ belongs to $C^{\infty}(\mathbb{R}^{n},\mathcal{S}(\mathbb{R}^{n}))$. Now, by continuing the above computation, we get
\begin{align*}
	\sup_{x\in\mathbb{R}^n}\sup_{t\in K\subset \mathbb{R}_{+}}\left|x^{\alpha}\partial^{\beta}_{x}\partial^{\gamma}_{t}\operatorname{W}_{0,t}u_{0}(x)\right|&=\sup_{x\in\mathbb{R}^n}\sup_{t\in K\subset \mathbb{R}_{+}}\left|x^{\alpha}\big(\partial^{\gamma}_{t}\mathcal{W}(\cdot;0,t)\ast \partial^{\beta}_{x}u_{0}\big)(x)\right|.
\end{align*}
Since $u_0 \in \mathcal{S}(\mathbb{R}^n)$, using Lemma \ref{avoidingstot}, we obtain
\begin{align}\label{lemma210}
	&\sup_{x \in \mathbb{R}^n} \Big| x^\alpha \big( \partial_t^\gamma \mathcal{W}(\cdot; s, t) \ast \partial_x^\beta u_0(\cdot) \big)(x) \Big| \notag\\
	&\lesssim_{\alpha, \beta, \gamma}
	\sum_{m \le \alpha} \sum_{\zeta \le \alpha-m} 
	\sum_{\mu'=1}^{|m|} \sum_{\mu=1}^\gamma 
	\left( \int_0^t \|a(\tau)\|_{\mathrm{op}} \, d\tau \right)^{\mu'} 
	\Theta_{\mu,\gamma}(t) 
	\sum_{|k'|\le N'}\sum_{|b|\le N}\|u_0\|_{k',b},\notag\\
	&\forall \alpha, \beta\in\mathbb{N}_0,~~ N > \mu'+\mu+|\beta|-|\zeta| + n, ~~
	N' > |\alpha-m-\zeta|+ N + n,
\end{align}
where
$$
\Theta_{\mu,\gamma}(t)\doteq
\left(\max_{1 \le q \le \gamma} \|\partial_t^{q} A(t)\|_{\mathrm{op}}\right)^{\mu},
$$

 and $\|\cdot\|_{k',b}$ denotes a Schwartz seminorm on $\mathcal{S}(\mathbb{R}^n)$.
Thus, taking supremum over $t \in K \subset \mathbb{R}_{+}$, we obtain
\begin{align*}
	&\sup_{x\in\mathbb{R}^n}\sup_{t\in K\subset \mathbb{R}_{+}}
	\left|x^{\alpha}\partial^{\beta}_{x}\partial^{\gamma}_{t}\operatorname{W}_{0,t}u_{0}(x)\right|\lesssim_{\alpha, \beta, \gamma, K}\notag\\
	&\sum_{m \le \alpha} \sum_{\zeta \le \alpha-m} 
	\sum_{\mu'=1}^{|m|} \sum_{\mu=1}^\gamma	\left(\sup_{\tau\in [0,\max K]}\|a(\tau)\|_{\mathrm{op}}\right)^{\mu'}
	\left(\sup_{t\in K}\max_{1 \le q \le \gamma} \|\partial_t^{q} A(t)\|_{\mathrm{op}}\right)^{\mu}\sum_{|k'|\le N'}\sum_{|b|\le N}\|u_0\|_{k',b}
	< \infty.\notag
\end{align*}
This completes the poof of \eqref{seminorm2}.

To prove \eqref{cntsatzero}, let $b\in\mathbb{N}$. Since $u_{0}\in\mathcal{S}(\mathbb{R}^n)$, similar to the proof of \eqref{seminorm2}, it follows from Lemma   \ref{avoidingstot} and continuity of the matrix $a$ at zero that
\begin{align}
	&\sup_{t\in[0,b]}\sup_{x\in\mathbb{R}^{n}} |x^\alpha\partial_x^{\beta} \operatorname{W}_{0,t} u_0(x)|\notag\\
	&\lesssim_{\alpha,\beta}
	\sum_{m\le \alpha}
	\sum_{\zeta\le \alpha-m}
	\sum_{\mu'=1}^{|m|}
	\sup_{t\in[0,b]}\left(\int\limits_{0}^{t}\| a(\tau)\|_{op}\,d\tau\right)^{\mu'}\, 
	\sum_{|b|\leq N}\sum_{|k'|\leq N'} 
	\|u_0\|_{k', b}\notag\\
	&\leq 	\sum_{m\le \alpha}
	\sum_{\zeta\le \alpha-m}
	\sum_{\mu'=1}^{|m|}
	\sup_{t\in[0,b]}\left(t\sup_{\tau\in[0,t]}\| a(\tau)\|_{op}\right)^{\mu'}\, 
	\sum_{|b|\leq N}\sum_{|k'|\leq N'} 
	\|u_0\|_{k', b}\notag\\
	&\lesssim_{b,\alpha}	\sum_{m\le \alpha}
	\sum_{\zeta\le \alpha-m}
	\sum_{\mu'=1}^{|m|}
	\left(\sup_{\tau\in[0,b]}\| a(\tau)\|_{op}\right)^{\mu'}\, 
	\sum_{|b|\leq N}\sum_{|k'|\leq N'} 
	\|u_0\|_{k', b}\\
	&\forall \alpha, \beta\in\mathbb{N}_0,~~ N > \mu'+\mu+|\beta|-|\zeta| + n, ~~
	N' > |\alpha-m-\zeta|+ N + n.\notag
\end{align}
To prove \eqref{seminorm4}, note that 
$\mathcal{W}(\cdot; s,t) \ast f(s) \in \mathcal{S}(\mathbb{R}^n)$ 
and the convolution converges locally uniformly in $x$. 
The equicontinuity of the family 
$\{\operatorname{W}_{s,t}\}_{(s,t)\in C}$ on compact sets 
$C \subset \overline{\Delta}$ as stated in Lemma \ref{avoidingstot}, ensures that
\[
[0,t] \ni s \longmapsto \operatorname{W}_{s,t}f(s)\in\mathcal{S}(\mathbb{R}^{n}),
\]
is continuous.
By the same argument, the mapping
\[
\overline{\Delta} \cap [0,\max K]^2 \ni (s,t) 
\longmapsto \operatorname{W}_{s,t}\partial^{\beta}_{x}f(s)
\in \mathcal{S}(\mathbb{R}^n),
\]
is jointly continuous. Hence, by \eqref{lemma21} in Lemma~\ref{avoidingstot}, we obtain

\begin{align}\label{AKMR2}
	&\sup_{x\in\mathbb{R}^n} \sup_{t\in K\subset \mathbb{R}_{+}}\left|x^{\alpha}\partial^{\beta}_{x}\partial^{\gamma}_{t}\int_0^t\operatorname{W}_{s,t}f(s,x)\,ds\right|\notag\\
	&=\sup_{x\in\mathbb{R}^n} \sup_{t\in K\subset \mathbb{R}_{+}}\left|x^{\alpha}\partial^{\beta}_{x}\partial^{\gamma}_{t}\int_0^tf(s)\ast\mathcal{W}(\cdot;s,t)(x)\,ds\right|\notag\\
	&=\sup_{x\in\mathbb{R}^n} \sup_{t\in K\subset \mathbb{R}_{+}}\left|x^{\alpha}\partial^{\beta}_{x}\Big[\gamma\partial^{\gamma-1}_{t}f(t,x)
	+\int_0^tf(s)\ast\partial_{t}^\gamma\mathcal{W}(\cdot;s,t)(x)\,ds\Big]\right|\notag\\
	&=\sup_{x\in\mathbb{R}^n}\sup_{t\in K\subset \mathbb{R}_{+}}\left|\gamma x^{\alpha}\partial^{\beta}_{x}\partial_{t}^{\gamma-1}f(t,x)
	+x^{\alpha}\partial^{\beta}_{x}\int_{0}^{t}f(s)\ast\partial_{t}^\gamma\mathcal{W}(\cdot;s,t)(x)\,ds\right|\notag\\
	&\leq\sup_{x\in\mathbb{R}^n} \sup_{t\in K\subset \mathbb{R}_{+}}\left|\gamma x^{\alpha}\partial^{\beta}_{x}\partial_{t}^{\gamma-1}f(t,x)\right|
	+\sup_{x\in\mathbb{R}^n} \sup_{t\in K\subset \mathbb{R}_{+}}\left| x^{\alpha}\partial^{\beta}_{x}\int_{0}^{t}f(s)\ast\partial_{t}^\gamma\mathcal{W}(\cdot;s,t)(x)\,ds\right|\notag\\	
	&\leq \sup_{x\in\mathbb{R}^n} \sup_{t\in K\subset \mathbb{R}_{+}}\left|\gamma x^{\alpha}\partial^{\beta}_{x}\partial_{t}^{\gamma-1}f(t,x)\right|
	+\sup_{x\in\mathbb{R}^n} \sup_{t\in K\subset \mathbb{R}_{+}} \int_{0}^{t} \left| x^{\alpha}\left( \partial^{\beta}_{x}f(s)\ast \partial_{t}^\gamma\mathcal{W}(\cdot;s,t)(x)\right)\right|\,ds\notag\\
	&\leq
	\gamma\sup_{t\in K}\sup_{x\in\mathbb{R}^{n}}
	\left|x^{\alpha}\partial_{x}^{\beta}\partial_{t}^{\gamma-1}f(t,x)\right|
	+\sup_{t\in K}\int_{0}^{t} \sup_{x\in\mathbb{R}^n}\left|x^{\alpha} \left( \partial_{t}^{\gamma}\mathcal{W}(\cdot;s,t)
	\ast \partial_x^{\beta} f(s,\cdot)\right)(x) \right| ds\displaybreak\notag\\
	&\lesssim_{\alpha,\beta,\gamma}\sup_{t\in K}\sup_{x\in\mathbb{R}^{n}}
	\left|x^{\alpha}\partial_{x}^{\beta}\partial_{t}^{\gamma-1}f(t,x)\right|\notag\\
	&+ \sum_{m \le \alpha} \sum_{\zeta \le \alpha-m} 
	\sum_{\mu'=1}^{|m|} \sum_{\mu=1}^\gamma 
	\left( \int_0^{\max (K)} \|a(\tau)\|_{\mathrm{op}} \, d\tau \right)^{\mu'} \sup_{t\in K}\Theta_{\mu,\gamma}(t)\,
	\int\limits_{0}^t I_{m,\zeta,\mu',\mu,\beta}(f(s)) ds \notag\\
	&\lesssim_{\alpha,\beta,\gamma}\sup_{t\in K}\sup_{x\in\mathbb{R}^{n}}
	\left|x^{\alpha}\partial_{x}^{\beta}\partial_{t}^{\gamma-1}f(t,x)\right|\notag\\
	&+ \max K \sum_{m \le \alpha} \sum_{\zeta \le \alpha-m} 
	\sum_{\mu'=1}^{|m|} \sum_{\mu=1}^\gamma 
	\left( \int_0^{\max (K)} \|a(\tau)\|_{\mathrm{op}} \, d\tau \right)^{\mu'} \sup_{t\in K}\Theta_{\mu,\gamma}(t)\,
    \max_{s\in[0,t]} I_{m,\zeta,\mu',\mu,\beta}(f(s)) < \infty,\\
    &\forall \alpha,\beta\in\mathbb{N}^{n}_{0},~~\forall \gamma\in\mathbb{N}_{0},~~\forall K\subset \mathbb{R}_{+} \text{compact}.\notag
\end{align}
The finiteness of the last line follows from the fact that 
$f\in C^{\infty}(\mathbb{R}_{+},\mathcal{S}(\mathbb{R}^{n}))$ and that the map
$$
s \mapsto I_{m,\zeta,\mu',\mu,\beta}(f(s)) 
\doteq \int\limits_{\mathbb{R}^n} (1 + |\xi|)^{\mu' + \mu + |\beta| - |\zeta|}
\big| \partial_\xi^{\alpha - m - \zeta} \hat{f}(s, \xi) \big| d\xi,
$$
is continuous on $[0,t]$, hence bounded.
To establish \eqref{cntsatzerointegral}, let $b \in \mathbb{N}$. Then we have
\begin{align}
	&\sup_{t\in[0,b]}\sup_{x\in\mathbb{R}^n} \left|x^{\alpha}\partial^{\beta}_{x}\int_0^t\operatorname{W}_{s,t}f(s,x) ds\right|
	\notag\\
	&\leq\sum_{m\le \alpha}
	\sum_{\zeta\le \alpha-m}
	\sum_{\mu'=1}^{|m|}\sup_{t\in[0,b]}t^{\mu'}\left(\sup_{\tau\in[0,t]}\|a(\tau)\|_{op}\right)^{\mu'}	\sum_{|b|\le N}\sum_{|k'|\le N'}t \sup_{s\in [0,t]}\|f(s)\|_{k',b}<\infty,\notag\\
	&\lesssim_{\alpha, b}\sum_{m\le \alpha}
	\sum_{\zeta\le \alpha-m}
	\sum_{\mu'=1}^{|m|}\left(\sup_{\tau\in[0,b]}\|a(\tau)\|_{op}\right)^{\mu'}	\sum_{|b|\le N}\sum_{|k'|\le N'} \sup_{s\in [0,b]}\|f(s)\|_{k',b}<\infty,\\
	&\forall \alpha, \beta\in\mathbb{N}_0, ~~N > \mu'+\mu+|\beta|-|\zeta| + n, ~~
	N' > |\alpha-m-\zeta|+N + n.\notag
\end{align}
Thus, the proof of \eqref{cntsatzerointegral} is complete.
To show that \eqref{solution} satisfies \eqref{cauchy}, by using part $5$ in Proposition \ref{pr2}, we may compute $\partial _{t}u(t,x)$ and $\mathcal{L}_tu(t,x)$ separately. Namely,
\begin{align}\label{utsatisfying}
	\partial _{t}u(t,x)&=\partial_{t}\operatorname{W}_{0,t}u_{0}(x)+\partial_{t}\left(\int\limits_0^t\operatorname{W}_{s,t}f(s,x)ds\right)\notag\\
	&=\partial_{t}(\mathcal{W}(\cdot;0,t)\ast u_{0})(x)+\operatorname{W}_{t,t}f(t,x)+\int\limits_0^t\partial_{t}\operatorname{W}_{s,t}f(s,x)ds\notag\\
	&=\left(\partial_{t}\mathcal{W}(\cdot;0,t)\ast u_0\right)(x)+f(t,x)+\int\limits_{0}^t\partial_{t}(\mathcal{W}(\cdot;s,t)\ast f(s,\cdot))(x)ds\notag\\
	&=\partial_{t}\mathcal{W}(\cdot;0,t)\ast u_{0}(x) +f(t,x)+\int\limits_{0}^t\partial_{t}\mathcal{W}(\cdot;s,t)\ast f(s,\cdot)(x)ds,\quad\forall x\in\mathbb{R}^n,~ \forall t\in\mathbb{R}_+.
\end{align}
To compute $\mathcal{L}_{t}u(t,x)$, we use Remark \ref{partialtimesx}. Accordingly, we can write
\begin{align}\label{ltsatisfying}
	\mathcal{L}_{t}u(t,x)&=\mathcal{L}_{t}(\operatorname{W}_{0,t}u_{0}(x))+\mathcal{L}_{t}\left(\int\limits_0^t\operatorname{W}_{s,t}f(s,x)ds\right)\notag\\
	&=\mathcal{L}_{t}(\mathcal{W}(\cdot;0,t)\ast u_{0}(x))+\mathcal{L}_{t}\left(\int\limits_0^t\mathcal{W}(\cdot;s,t)\ast f(s,\cdot)(x)ds\right)\notag\\
	&=\sum_{i,j=1}^{n} a_{ij}(t)\partial^2_{x_ix_j}(\mathcal{W}(\cdot;0,t)\ast u_{0}(x))+\sum_{i,j=1}^{n} a_{ij}(t)\partial^2_{x_ix_j}\left(\int\limits_0^t\mathcal{W}(\cdot;s,t)\ast f(s,\cdot)(x)ds\right)\notag\\
	&=\sum_{i,j=1}^{n} a_{ij}(t)(\partial^2_{x_ix_j}\mathcal{W}(\cdot;0,t))\ast u_{0}(x)+\sum_{i,j=1}^{n} a_{ij}(t)\int\limits_0^t\partial^2_{x_ix_j}\left( \mathcal{W}(\cdot;s,t)\ast f(s,\cdot)(x)\right) ds\notag\\
	&=\left(\sum_{i,j=1}^{n} a_{ij}(t)\partial^2_{x_ix_j}\mathcal{W}(\cdot;0,t)\ast u_{0}\right)(x)+\int\limits_0^t\sum_{i,j=1}^{n} a_{ij}(t)(\partial^2_{x_ix_j}\mathcal{W}(\cdot;s,t))\ast f(s,\cdot)(x)ds\notag\\
	&=\mathcal{L}_{t}\mathcal{W}(\cdot,0,t)\ast u_{0}(x)+\int\limits_0^t\mathcal{L}_{t}\mathcal{W}(\cdot,s,t)\ast f(s,\cdot)(x)ds,\quad\forall x\in\mathbb{R}^n,~~ \forall t\in\mathbb{R}_+.
\end{align}
Now, by using \eqref{utsatisfying} and \eqref{ltsatisfying}, we conclude that
\begin{align}
	\partial_{t}u(t,x)-\mathcal{L}_{t}u(t,x)&=\partial_{t}\mathcal{W}(\cdot;0,t)\ast u_{0}(x) +f(t,x)+\int\limits_{0}^t\partial_{t}\mathcal{W}(\cdot;s,t)\ast f(s,\cdot)(x)ds\notag\\
	&-\mathcal{L}_{t}\mathcal{W}(\cdot,0,t)\ast u_{0}(x)-\int\limits_0^t\mathcal{L}_{t}\mathcal{W}(\cdot,s,t)\ast f(s,\cdot)(x)ds\notag\\
	&=(\partial_{t}\mathcal{W}(\cdot;0,t)-\mathcal{L}_{t}\mathcal{W}(\cdot;0,t))\ast u_{0}(x)+f(t,x)\notag\\
	&+\int\limits_{0}^t(\partial_{t}\mathcal{W}(\cdot;s,t)-\mathcal{L}_{t}\mathcal{W}(\cdot;s,t))\ast f(s,\cdot)(x)ds\notag\\
	&=f(t,x), \quad\forall x\in\mathbb{R}^n,~\forall t\in\mathbb{R}_+.
\end{align}
The last line follows from part 2 of Proposition \ref{pr1}. Additionally, we have 
\begin{equation*}
	\lim\limits_{t\to0+}u(t,\cdot)=u_0, \quad\forall t\in \mathbb{R}_{+}.
\end{equation*}
The reason is that when we take limit from \eqref{solution} as $t$ goes to $0+$, using part $5$ of Proposition \ref{pr2}, we have 
\begin{align}
	\lim_{t\rightarrow 0^{+}}u(t,\cdot)&=\lim_{t\rightarrow 0^{+}}\operatorname{W}_{0,t}u_0+\lim_{t\rightarrow 0^{+}}\int\limits_0^t\operatorname{W}_{s,t}f(s,\cdot)ds =u_{0},\quad\forall t\in\mathbb{R}_+.
\end{align} 
To verify the non-negativity property of the solution $u$,  we note that positivity of $$\mathcal{W}(x; s, t)>0,~~ \forall x \in \mathbb{R}^n, ~ \forall(s, t) \in \Delta,$$ as well as non-negativity of $u_{0}$ and $f$ imply the following: 
$$
\operatorname{W}_{0,t} u_{0}(x) = \left(\mathcal{W}(\cdot; 0, t)\ast u_0\right) (x) \geq 0,\quad\forall x\in\mathbb{R}^n,~~\forall t\in\mathbb{R}_{+},
$$
as well as  
$$
\int\limits_{0}^{t} \operatorname{W}_{s,t} f(s, x) ds = \int\limits_{0}^{t} \mathcal{W}(\cdot; s, t) \ast f(s)(x) ds \geq 0,\quad\forall x\in\mathbb{R}^n,~~\forall t\in\mathbb{R}_{+}.
$$
Hence, by \eqref{solution} the solution $u$  remains non-negative. Thus, the proof is complete. 
\end{proof}

We now derive a collection of estimates, including the energy estimate and $L^p$-$L^q$
bounds, which will be used in subsequent arguments.

\begin{lemma}[Energy estimate]\label{energyestim}
Let $1<q<\infty$, and let $u \in C^{1}(\mathbb{R}_{+}, L^{q}(\mathbb{R}^{n}))\,\cap\, C([0,\infty), L^q(\mathbb{R}^n))\,\cap\, C(\mathbb{R}_{+}, W^{q,2}(\mathbb{R}^{n}))$, satisfy the homogeneous heat equation 
\begin{equation}\label{homheat}
	\partial_t u(t,x)-\mathcal{L}u(t,x)= 0,
\end{equation}

with initial data $u(0,\cdot)=u_0\in L^q(\mathbb{R}^n)$. Define the following energy functional on $L^q(\mathbb{R}^n)$:
\begin{equation}\label{energy}
	E(v)\doteq \|v\|^{q}_{q}=\int\limits_{\mathbb{R}^n}|v(x)|^{q}dx.
\end{equation}
	Then, the mapping  $t\mapsto E(u(t)),$ is non-increasing on $\mathbb{R}_+$. In particular, the following inequality holds:
\begin{align}\label{energy12}
	\|u(t,\cdot)\|_q\leq \|u_0\|_q,\quad\forall t\in\mathbb{R}_+.
\end{align}
\end{lemma}
\begin{proof}
Let $1<q<\infty$. Since $u \in C^1(\mathbb{R}_+, L^q(\mathbb{R}^n))$, the mapping $t\mapsto E(u(t))$ is differentiable and
$$
\frac{d}{dt}E(u(t))=q\int\limits_{\mathbb{R}^n}|u(t,x)|^{q-2}u(t,x)\,\partial_t u(t,x)\,dx,~~\forall t\in \mathbb{R}_{+}.
$$ 
Now, using the homogeneous heat equation \eqref{homheat}, we obtain
\begin{align}\label{energyestimate1}
	&q\int\limits_{\mathbb{R}^n}|u(t,x)|^{q-2}u(t,x)\partial_{t}u(t,x)dx-q\sum_{i,j=1}^n a_{i,j}(t)\int\limits_{\mathbb{R}^n}|u(t,x)|^{q-2}u(t,x)\frac{\partial^{2}}{\partial x_{i}\partial x_{j}}u(t,x) dx =0,\quad\forall t\in\mathbb{R}_+.
\end{align}
We see that the first term in \eqref{energyestimate1}, is exactly $\frac{d}{dt}E(u(t))$. To address the second term on the left hand side of \eqref{energyestimate1}, using integration by parts, we have
\begin{equation}
	\int\limits_{\mathbb{R}^n}|u(t,x)|^{q-2}u(t,x)\frac{\partial^{2}}{\partial x_{i}\partial x_{j}}u(t,x)dx=-\int\limits_{\mathbb{R}^n}\frac{\partial}{\partial x_{i}}u(t,x)\frac{\partial}{\partial x_{j}}\left(|u(t,x)|^{q-2}u(t,x)\right)dx,\quad\forall t\in\mathbb{R}_+.
\end{equation}
Since 
\begin{equation*}
	\frac{\partial}{\partial x_{j}}\left(|u(t,x)|^{q-2}u(t,x)\right)=(q-1)|u(t,x)|^{q-2}\frac{\partial}{\partial x_{j}}u(t,x),~~\forall x\in\mathbb{R}^{n},~~\forall t\in \mathbb{R}_{+},
\end{equation*}
we have
\begin{align}
	&\sum_{i,j=1}^na_{i,j}(t)\int\limits_{\mathbb{R}^n}|u(t,x)|^{q-2}u(t,x)\frac{\partial^{2}}{\partial x_{i}\partial x_{j}}u(t,x)dx\notag\\
	&=-(q-1)\sum_{i,j=1}^n a_{i,j}(t)\int\limits_{\mathbb{R}^n}\frac{\partial}{\partial{x_i}}u(t,x)|u(t,x)|^{q-2}\frac{\partial}{\partial{x_j}}u(t,x)dx,\quad\forall t\in\mathbb{R}_+.
\end{align}
It is worth noting that in the integration by parts calculation, we used the theorem in Exercise 1.1.16 of \cite{Ruzhansky} and the fact that $u\in \mathcal{D}om(\mathcal{L}_t)$.
Therefore, since the matrix $a(t)$ is positive definite for every $t\in\mathbb{R}_+$,
\begin{align}\label{nonincreasingenergy1}
	\frac{d}{dt} E(u(t))=-q(q-1)\sum_{i,j=1}^n a_{i,j}(t)\int\limits_{\mathbb{R}^n}\frac{\partial}{\partial{x_i}}u(t,x)|u(t,x)|^{q-2}\frac{\partial}{\partial{x_j}}u(t,x)dx\leq0,\quad\forall t\in\mathbb{R}_+.
\end{align}
Thus, the energy function is non-increasing.
Integrating both sides of \eqref{nonincreasingenergy1} with respect to $t$, we have
\begin{equation}
	E(u(t))-E(u_{0})\leq 0, ~\forall t\in\mathbb{R}_+.
\end{equation}
Hence,
\begin{align}\label{nonincreasingenergy}
	\|u(t,\cdot)\|_q\leq \|u_0\|_q,\quad\forall t\in\mathbb{R}_+,~\forall~1<q<\infty.
\end{align}
This completes the proof.
\end{proof}

To have $L^p$-$L^q$ estimates for the operator family $\{\operatorname{W}_{s,t}\}_{(s,t)\in\Delta}$, we need the following remark.

\begin{remark}\label{normpkernel}
For every $(s,t)\in\Delta$, we have the identity
$$
\|\mathcal{W}(\cdot;s,t)\|_p=\frac1{p^{\frac{n}{2p}}\left((4\pi)^n\det[A(t)-A(s)]\right)^{\frac{p-1}{2p}}},\quad\forall p\in[1,+\infty],
$$
where $p=+\infty$ is understood as the limit,
$$
\|\mathcal{W}(\cdot;s,t)\|_\infty=\frac1{\sqrt{(4\pi)^n\det[A(t)-A(s)]}}.
$$
This can be easily obtained in view of
$$\|\mathcal{W}(\cdot;s,t)\|_p^p=\int\limits_{\mathbb{R}^n} \dfrac{1}{\sqrt{(4\pi)^{np}(\det[A(t)-A(s)])^p}}e^{-\frac{p}{4}\langle x,[A(t)-A(s)]^{-1}x\rangle}dx,
$$
using the calculation in the proof of part $1$ of Proposition \ref{pr1}. 
\end{remark}
This leads to the following estimates.

\begin{proposition}\label{pr4}
Let $1\le p,q,r<+\infty$ be such that
$$
\frac1p+\frac1q=\frac1r+1.
$$
\begin{itemize}
\item[1.] For $\forall(s,t)\in\overline{\Delta}$, the operator $\operatorname{W}_{s,t}$ with domain $\mathcal{S}(\mathbb{R}^n)$ extends to a bounded operator $$\operatorname{W}_{s,t}:L^q(\mathbb{R}^n)\to L^r(\mathbb{R}^n),$$ with operator norm
\begin{equation}\label{estimationws,t}
	\|\operatorname{W}_{s,t}\|_{q\to r}\le\frac1{p^{\frac{n}{2p}}\left((4\pi)^n\det[A(t)-A(s)]\right)^{\frac{p-1}{2p}}}.
\end{equation}
Here, $\operatorname{W}_{t,t}=\mathrm{1}$ is understood for $\forall t\in[0,+\infty)$;

\item[2.] For $\forall v\in L^q(\mathbb{R}^n)$, the map $(s,t)\mapsto\operatorname{W}_{s,t}v$ belongs to
$$
C^\infty(\Delta,L^r(\mathbb{R}^n))\,\cap\, C(\overline{\Delta},L^r(\mathbb{R}^n));
$$

\item[3.] For $\forall(s,t)\in\Delta$,
$$
\operatorname{W}_{s,t}L^q(\mathbb{R}^n)\subset W^{q,\infty}(\mathbb{R}^n);
$$

\item[4.] For $\forall v\in L^q(\mathbb{R}^n)$,
$$
\partial_t\operatorname{W}_{s,t}v-\mathcal{L}_t\operatorname{W}_{s,t}v=0.
$$
\end{itemize}
\end{proposition}
\begin{proof} 
To prove part $1$, we have
\begin{eqnarray*}
	\operatorname{W}_{s,t}&: \mathcal{S}(\mathbb{R}^n)\subset L^q(\mathbb{R}^n) \longrightarrow \mathcal{S}(\mathbb{R}^n)\subset L^{r}(\mathbb{R}^n),\\
	&v \longmapsto \operatorname{W}_{s,t}v, \quad \forall ~1 \leq q, r < \infty.
\end{eqnarray*}
Let $v\in\mathcal{S}(\mathbb{R}^n)\subset L^{q}(\mathbb{R}^n)$ for every $1 \leq q < \infty$. In addition, we have
$\mathcal{W}(\cdot;s,t)\in \mathcal{S}(\mathbb{R}^n) \subset L^{p}(\mathbb{R}^n)$, for every $1\leq p < \infty$. Therefore, by applying Young's inequality for convolutions, we obtain
\begin{align}\label{Young1}
	\|\operatorname{W}_{s,t}v\|_{r} = \|\mathcal{W}(\cdot;s,t) \ast v \|_{r} \leq \|\mathcal{W}(\cdot;s,t)\|_{p} \|v\|_{q}<\infty,~ \text{where}, 
	~ 1\leq p,q,r < \infty, \text{and} ~ \frac{1}{p}+\frac{1}{q}=1+\frac{1}{r}.
\end{align}
Therefore using \eqref{Young1}, and Remark \ref{normpkernel} we have
\begin{align}
	\left\|\operatorname{W}_{s,t}\right\|_{q\rightarrow r}&=
	\sup_{\| v \|_{q}=1} \left\|\operatorname{W}_{s,t}v\right\|_{r}\leq\sup_{\| v \|_{q}=1}\left\|\mathcal{W}(\cdot;s,t)\right\|_{p} \|v\|_{q}\notag\\
	&=\left\|\mathcal{W}(\cdot;s,t)\right\|_{p} = \left((4\pi)^n \det[A(t)-A(s)]\right)^{\frac{1-p}{2p}} p^{-\frac{n}{2p}},~~1\leq p<\infty.\notag
\end{align}
In addition, if $p=\infty$, then
\begin{align*}
	\left\|\operatorname{W}_{s,t}\right\|_{q\rightarrow r} \leq \dfrac{1}{\sqrt{(4\pi)^n \det[A(t)-A(s)]}}, \quad 1\leq q,r \leq \infty.
\end{align*}
Thus, the proof of part $1$ is complete. 

To begin proving part~2, we first show that for every $v \in L^q(\mathbb{R}^n)$, the mapping
\[
(s,t)\mapsto \operatorname{W}_{s,t}v,
\]
belongs to $C^{\infty}(\Delta, L^r(\mathbb{R}^n))$.
Since $\mathcal{W}\in C^{\infty}(\Delta, \mathcal{S}(\mathbb{R}^n))$, it follows that for every pair of multi-indices
$\gamma,\eta \in \mathbb{N}_0$ (with respect to $s$ and $t$, respectively),
\[
\partial_s^\gamma \partial_t^\eta \mathcal{W}(\cdot; s,t)\in \mathcal{S}(\mathbb{R}^n)\subset L^1(\mathbb{R}^n),\quad \forall (s,t)\in \Delta.
\]

Hence, the convolution $(\partial_s^\gamma \partial_t^\eta \mathcal{W}(\cdot; s,t)) \ast v$ where $(s,t)\in \Delta$ is well-defined in $L^r(\mathbb{R}^n)$ and depends continuously on $(s,t)$, since convolution
\[
\mathcal{S}(\mathbb{R}^n)\times L^q(\mathbb{R}^n)\to L^r(\mathbb{R}^n),
\]
is a continuous bilinear map.
Moreover, differentiation commutes with convolution, so that
\[
\partial_s^\gamma \partial_t^\eta \operatorname{W}_{s,t}v
= (\partial_s^\gamma \partial_t^\eta \mathcal{W}(\cdot; s,t)) \ast v,\quad \forall (s,t)\in \Delta.
\]
Therefore, $(s,t)\mapsto \operatorname{W}_{s,t}v$ belongs to $C^\infty(\Delta,L^r(\mathbb{R}^n))$.

For continuity up to the boundary $\overline{\Delta}$, we restrict to the case $q = r$, which implies $p = 1$.	We need to show that $\lim_{t \rightarrow s} \|\operatorname{W}_{s,t}v - v\|_{r} = 0$, meaning $\operatorname{W}_{s,t}$ strongly converges to the identity operator $\mathrm{1}$ as $t \to s$. We assume $r < +\infty$.
Let $\epsilon > 0$. Since $\mathcal{S}(\mathbb{R}^n)$ is dense in $L^r(\mathbb{R}^n)$ for $r < \infty$, we can find a Schwartz function $\phi \in \mathcal{S}(\mathbb{R}^n)$ such that:
$$
	\|v - \phi\|_{r} < \frac{\epsilon}{3}.
$$
By part 5 of Proposition \ref{pr2}, as $t \rightarrow s$, $\operatorname{W}_{s,t}\phi \to \phi$ in $\mathcal{S}(\mathbb{R}^n)$, and by Lemma 2.9 of \cite{wong} it follows that the convergence holds in $L^r(\mathbb{R}^n)$. Thus, there exists a $\delta > 0$ such that for all $t - s < \delta$:
$$
\|\operatorname{W}_{s,t}\phi - \phi\|_{r} < \frac{\epsilon}{3},\quad \forall (s,t)\in \Delta.
$$
Using Young's inequality with $p=1$, we have 
$$
\|\operatorname{W}_{s,t}(v-\phi)\|_{r} \leq \|\mathcal{W}(\cdot; s,t)\|_1 \|v-\phi\|_{r} = \|v-\phi\|_{r},\quad \forall (s,t)\in \Delta.
$$
Therefore, one can obtain:
\begin{align*}
	\|\operatorname{W}_{s,t}v - v\|_{r} &= \|\operatorname{W}_{s,t}(v-\phi) + (\operatorname{W}_{s,t}\phi - \phi) + (\phi - v)\|_{r} \\
	&\le \|\operatorname{W}_{s,t}(v-\phi)\|_{r} + \|\operatorname{W}_{s,t}\phi - \phi\|_{r} + \|\phi - v\|_{r} \\
	&\le \|v - \phi\|_{r} + \|\operatorname{W}_{s,t}\phi - \phi\|_{r} + \|\phi - v\|_{r} \\
	&< \frac{\epsilon}{3} + \frac{\epsilon}{3} + \frac{\epsilon}{3} = \epsilon.
\end{align*}
This establishes strong continuity at the boundary $s=t$, proving that $(s,t) \mapsto \operatorname{W}_{s,t}v \in C(\overline{\Delta}, L^r(\mathbb{R}^n))$ for $v \in L^r(\mathbb{R}^n)$. Thus, the proof of part 2 is complete. 

To prove part 3, fix $(s,t) \in \Delta$ and let $v \in L^q(\mathbb{R}^n)$. By definition, $\operatorname{W}_{s,t}v = \mathcal{W}(\cdot; s,t) \ast v$. 
Since $\mathcal{W}(\cdot; s,t) \in \mathcal{S}(\mathbb{R}^n)$, it follows that for any multi-index $\alpha \in \mathbb{N}_0^n$, the partial derivative $\partial_x^\alpha \mathcal{W}(\cdot; s,t)$ exists and is also in $\mathcal{S}(\mathbb{R}^n)$. In particular, $\partial_x^\alpha \mathcal{W}(\cdot; s,t) \in L^1(\mathbb{R}^n)$. Using Young's convolution inequality with $p=1$, we obtain
\begin{equation*}
	\|\partial_x^\alpha (\operatorname{W}_{s,t}v)\|_{q} = \|(\partial_x^\alpha \mathcal{W}(\cdot; s,t)) \ast v\|_{q} \le \|\partial_x^\alpha \mathcal{W}(\cdot; s,t)\|_{1} \|v\|_{q}<\infty, \quad\forall \alpha\in\mathbb{N}_0^n, ~\forall (s,t)\in\Delta.
\end{equation*}
This shows that $\operatorname{W}_{s,t}v \in W^{q, \infty}(\mathbb{R}^n)$, and hence
$$
\operatorname{W}_{s,t}L^q(\mathbb{R}^n) \subset W^{q, \infty}(\mathbb{R}^n),\quad \forall(s,t)\in\Delta.
$$
Thus, we complete the proof of part 3.

To prove part $4$, let $v \in L^q(\mathbb{R}^n)$ and fix $(s,t) \in \Delta$. By part $1$ of Proposition~\ref{pr4}, the operator $\operatorname{W}_{s,t}$ extends to a bounded linear operator on $L^q(\mathbb{R}^n)$, and by part $2$ of Proposition~\ref{pr4}, the map
\[
(s,t) \mapsto \operatorname{W}_{s,t}v,
\]
belongs to $C^\infty(\Delta, L^r(\mathbb{R}^n))$, where $\frac{1}{r}=\frac{1}{p}+\frac{1}{q}-1$.
In particular, $\partial_t \operatorname{W}_{s,t}v \in L^r(\mathbb{R}^n)$ exists.

We now identify this derivative. By part $3$ of Proposition~\ref{pr4} applied with $\gamma=(0,1)$,
\[
\partial_t \operatorname{W}_{s,t}v
=
(\partial_t \mathcal{W}(\cdot;s,t)) \ast v
\quad \text{in } L^r(\mathbb{R}^n),\quad \forall (s,t)\in \Delta.
\]

Similarly, applying spatial derivatives as in Remark \ref{partialtimesx} and using the definition of $\mathcal{L}_t$, we obtain
\[
\mathcal{L}_t \operatorname{W}_{s,t}v
=
(\mathcal{L}_t \mathcal{W}(\cdot;s,t)) \ast v
\quad \text{in } L^r(\mathbb{R}^n),\quad \forall (s,t)\in \Delta.
\]

By part $2$ of Proposition~\ref{pr1}, the kernel satisfies
\[
\partial_t \mathcal{W}(x;s,t) - \mathcal{L}_t \mathcal{W}(x;s,t) = 0,
\quad \forall x\in\mathbb{R}^n,\ \forall (s,t)\in\Delta.
\]

Since $\partial_t \mathcal{W}(\cdot;s,t),\, \mathcal{L}_t \mathcal{W}(\cdot;s,t)\in L^p(\mathbb{R}^n)$, we may convolve this identity with $v\in L^q(\mathbb{R}^n)$ to obtain
\[
\partial_t \operatorname{W}_{s,t}v
-
\mathcal{L}_t \operatorname{W}_{s,t}v
=
\bigl(\partial_t \mathcal{W}(\cdot;s,t)
-
\mathcal{L}_t \mathcal{W}(\cdot;s,t)\bigr) \ast v=0,\quad \forall (s,t)\in \Delta,
\]
in $L^r(\mathbb{R}^n)$. Thus, the proof of the proposition is complete.
\end{proof}

Now, we are ready to state the well-posedness of the $L^q$-Cauchy problem of our heat equation.

\begin{proposition}\label{pr5}
Let $u_0\in L^q(\mathbb{R}^n)$ and $f\in C^{\infty}(\mathbb{R}_+,W^{q, \infty}(\mathbb{R}^n))\,\cap\, C([0,\infty),W^{q,\infty}(\mathbb{R}^{n}))$, for some $q\in[1,+\infty)$. Then the function $u:\mathbb{R}_+\times\mathbb{R}^n\to\mathbb{C}$ defined by
\begin{equation}\label{solution1}
	u(t)\doteq\operatorname{W}_{0,t}u_0+\int\limits_0^t\operatorname{W}_{s,t}f(s)ds,\quad\forall t\in\mathbb{R}_+,
\end{equation}
 is a  solution of the Cauchy problem
\begin{align}\label{cauchypr50}
\begin{cases}
	\partial_t u(t)-\mathcal{L}_tu(t)=f(t),\quad\forall t\in\mathbb{R}_+,\\
	\lim\limits_{t\to0+}u(t)=u_0,
\end{cases}
\end{align}
that satisfies
$$
u\in C^\infty(\mathbb{R}_+,W^{q, \infty}(\mathbb{R}^n))\,\cap\,C([0,+\infty),L^q(\mathbb{R}^n)),
$$
and 
\begin{equation}\label{normestimatesimple}
	\|u(t)\|_q\le\|u_0\|_q+\int\limits_0^t\|f(s)\|_qds,\quad\forall t\in[0,+\infty).
\end{equation} 
Moreover, \eqref{solution1} is the unique solution of the  Cauchy problem \eqref{cauchypr50} for $q\in(1,\infty)$.
Furthermore, if $u_0\ge0$ and $f\ge0$ then $u\ge0$.
\end{proposition}
\begin{proof}
First, we verify that the function $u$ in \eqref{solution1} satisfies the Cauchy problem \eqref{cauchypr50}. We first consider derivatives with respect to time.
	Since $\mathcal{S}(\mathbb{R}^n)$ is continuously embedded into $L^q(\mathbb{R}^n)$ and $W^{q,\infty}(\mathbb{R}^n)$, the time-differentiability established in Proposition \ref{pr3} for $\mathcal{S}(\mathbb{R}^n)$ remains valid in these spaces as well. Hence, all time derivatives appearing in \eqref{cauchypr50} are well-defined in $L^q(\mathbb{R}^n)$ and $W^{q,\infty}(\mathbb{R}^n)$.
	
	Next, we treat spatial derivatives. Fix $\beta \in \mathbb{N}_0^n$.
	Since
	$
	f \in C([0,\infty), W^{q,\infty}(\mathbb{R}^n)),
	$
	it follows that the map
	$
	s \mapsto \partial_x^\beta f(s)
	$
	is continuous from $[0,t]$ into $L^q(\mathbb{R}^n)$. In particular,
	$$
	\sup_{s\in[0,t]} \|\partial_{x}^\beta f(s)\|_{q} < \infty.
	$$
	By Lemma 9.1 from \cite{Brezis}, we have 
	$$
	\partial_x^\beta (f(s)\ast\mathcal{W}(\cdot;s,t))= (\partial_x^\beta f(s)\ast\mathcal{W}(\cdot;s,t)),
	$$
	and by Young's inequality,
	$$
	\|\partial_x^\beta (f(s)\ast\mathcal{W}(\cdot;s,t))\|_{q}
	\leq \|\partial_x^\beta f(s)\|_{q} \|\mathcal{W}(\cdot;s,t)\|_{1}	\leq \|\partial_x^\beta f(s)\|_{q}
	\leq C_\beta,
	$$
	uniformly for $s \in [0,t]$.
	Therefore, the map
	$$
	s \mapsto \partial_x^\beta \operatorname{W}_{s,t} f(s),\quad\forall(s,t)\in\Delta,
	$$
	belongs to $C([0,t], L^q(\mathbb{R}^n))$, and in particular is integrable. Hence,
	$$
	\int_0^t \partial_x^\beta \operatorname{W}_{s,t} f(s) ds,\quad\forall(s,t)\in\Delta,
	$$
	is well-defined in $L^q(\mathbb{R}^n)$.
	Therefore, we may pass the derivative under the integral sign to obtain
	$$
	\partial_x^\beta \int_0^t \operatorname{W}_{s,t} f(s) ds
	= \int_0^t \partial_x^\beta \operatorname{W}_{s,t} f(s) ds,\quad\forall(s,t)\in\Delta.
	$$
Applying Young's inequality, we then obtain
\begin{align}\label{maximalregularity1}
	&\left\| \int\limits_0^t \partial_x^\beta\left( \operatorname{W}_{s,t} f(s) \right)ds \right\|_q =\left\| \int\limits_0^t \partial_x^\beta\left( \mathcal{W}(\cdot; s,t)\ast f(s) \right) ds \right\|_q
	=\left\| \int\limits_0^t  \mathcal{W}(\cdot; s,t)\ast \partial_x^\beta f(s)  ds \right\|_q\notag\\
	&\leq \int\limits_0^t \left\|\mathcal{W}(\cdot; s,t)\ast \partial_x^\beta f(s) \right\|_q  ds\leq \int\limits_0^t \| \partial_x^\beta f(s) \|_q  \left\| \mathcal{W}(\cdot; s,t) \right\|_1  ds= \sup_{s \in [0,t]} \| \partial_x^\beta f(s) \|_q  t< \infty.
	\end{align}	
	Taking into account the above arguments, and combining them with the time-regularity result established in Proposition~\ref{pr3}, we conclude that all the terms appearing in the representation of $u$ are sufficiently regular in $L^q(\mathbb{R}^n)$.
	
	Consequently, the computations in Proposition~\ref{pr3} remain valid when interpreted in $L^q(\mathbb{R}^n)$, and we deduce that $u$ satisfies the heat equation in \eqref{cauchypr50}. Moreover, we verify the initial condition. By definition, we have
	\[
	u(0,\cdot)=u_0.
	\]
	Furthermore, by part~2 of Proposition~\ref{pr4}, it holds that
	\begin{equation}\label{contntyatzero}
		\lim_{t \to 0^{+}} \left\| \operatorname{W}_{0,t}u_{0} - u_{0} \right\|_{q} = 0.
	\end{equation}
	Consequently,
	\[
	\lim_{t \to 0^{+}} u(t) = u_0, \quad \text{in } L^q(\mathbb{R}^n).
	\]
Now, we want to prove that 
\begin{equation}\label{sopr5}
	u \in C^{\infty}(\mathbb{R}_{+},W^{q,\infty}(\mathbb{R}^{n})),\quad \text{for} ~~1\leq q <\infty.
\end{equation}

At first, we fix $t\in\mathbb{R}_+$, and we show that $u(t)\in W^{q,\infty}(\mathbb{R}^{n})$.
By part 3 of Proposition \ref{pr4} and the assumption $u_0 \in  L^q(\mathbb{R}^n)$, we have 
$\operatorname{W}_{0,t} u_0 \in W^{q,\infty}(\mathbb{R}^n)$. Similarly, for all $(s,t) \in \Delta$,  
\[
\operatorname{W}_{s,t} f(s) \in W^{q,\infty}(\mathbb{R}^n),
\] 
because $f(s) \in W^{q,\infty}(\mathbb{R}^n) \subset L^q(\mathbb{R}^n)$.  
Next, following \eqref{maximalregularity1}, and its argument, we deduce that 
$$
\int\limits_0^t \operatorname{W}_{s,t} f(s) \, ds \in W^{q,\infty}(\mathbb{R}^n),\quad \forall t\in \mathbb{R}_{+}.
$$ 
Hence,
\[
u(t) = \operatorname{W}_{0,t} u_0 + \int\limits_0^t \operatorname{W}_{s,t} f(s) \, ds \in W^{q,\infty}(\mathbb{R}^n),\quad \forall t\in \mathbb{R}_{+}.
\]

To continue proving \eqref{sopr5}, we proceed as follows. Since 
$u(t) \in W^{q,\infty}(\mathbb{R}^{n})$ and the coefficient matrix 
\(a(t)\) is smooth for all 
\(t \in \mathbb{R}_{+}\), we have
\begin{equation}\label{preserve}
	\mathcal{L}_t\bigl(W^{q,\infty} (\mathbb{R}^{n})\bigr) \subset 
	W^{q,\infty}(\mathbb{R}^{n}), \quad \forall t \in \mathbb{R}_{+}.
\end{equation}
Therefore, $	\mathcal{L}_t u(t)\in 	W^{q,\infty}(\mathbb{R}^{n}), \quad \forall t \in \mathbb{R}_{+}$.
The reason is that $\mathcal{L}_t$ is a second order differential operator that maps the space $W^{q,k}(\mathbb{R}^{n})$ to the space $W^{q,k-2}(\mathbb{R}^{n})$ for every $q\in (1,\infty]$ and $k\in\mathbb{N}$. Therefore, by the definition of the Sobolev space, we have
\begin{align}
	\mathcal{L}_t(	W^{q,\infty}(\mathbb{R}^{n})) &= \mathcal{L}_t(\bigcap_{k=0}^{\infty}	W^{q,k}(\mathbb{R}^{n})) \subseteq \bigcap_{k=0}^{\infty} \mathcal{L}_t (W^{q,k}(\mathbb{R}^{n}))\notag\\
	&=\bigcap_{k=0}^{\infty}W^{q,k-2}(\mathbb{R}^{n}) =\bigcap_{k=0}^{\infty}  W^{q,k}(\mathbb{R}^{n})=W^{q,\infty}(\mathbb{R}^{n}) .
\end{align}

Then, considering equation \eqref{cauchypr50} and using \eqref{preserve} 
together with the assumption $f(t) \in W^{q,\infty}(\mathbb{R}^{n})$ 
for all $t \in \mathbb{R}_{+}$, it immediately follows that
$
\partial_t u(t) \in W^{q,\infty}(\mathbb{R}^n)$, for every $t \in \mathbb{R}_{+}$.

Now, for every $t\in\mathbb{R}_+$ we
differentiate both sides of the equation in \eqref{cauchypr50} with respect to time. It yields
\begin{equation}\label{2}
	\partial_t^2 u(t) - \partial_t \mathcal{L}_t u(t) = \partial_t f(t), 
	\quad \forall t \in \mathbb{R}_{+}.
\end{equation}
Using the fact that for every fixed $t\in\mathbb{R}_+$, $u(t), \partial_t u(t) \in W^{q,\infty}(\mathbb{R}^n)$ and applying the product rule together with \eqref{preserve}, we obtain

\begin{equation}\label{eq104}
	\partial_t \mathcal{L}_t u(t) = \dot{\mathcal{L}}_t u(t) + \mathcal{L}_t \partial_t u(t)\in W^{q,\infty}(\mathbb{R}^{n}),
\end{equation}
where 
\begin{align*}
	\dot{\mathcal{L}}_t u(t) \doteq \sum_{i,j=1}^n \dot{a}_{ij}(t) \, \partial_{x_i x_j}^2 u(t).
\end{align*}

Hence, by \eqref{eq104}, and the fact that
$f \in C^{\infty}(\mathbb{R}_{+}, W^{q,\infty}(\mathbb{R}^n))$, it follows from \eqref{2} that 
$\partial_t^2 u(t) \in W^{q,\infty}(\mathbb{R}^{n})$.
After differentiating \eqref{cauchypr50} $k$ times with respect to $t$, we obtain the following:

\begin{equation}\label{hypothesis}
	\partial_t^{k+1} u(t) - \partial_t^k \mathcal{L}_t u(t) = \partial_t^k f(t), \quad \forall t \in \mathbb{R}_{+},
\end{equation}
and
\begin{equation}\label{repeat}
	\partial_{t}^{m} u(t) \in W^{q,\infty}(\mathbb{R}^{n}), \quad \forall t \in \mathbb{R}_{+}, \quad \forall m = 0,1,\dots,k+1.
\end{equation}
Therefore, $u \in C^\infty(\mathbb{R}_{+}, W^{q,\infty}(\mathbb{R}^{n}))$. Now, we want to show that $u \in C([0,\infty), L^q(\mathbb{R}^n))$. To do so, considering $C^{\infty}(\mathbb{R}_+,W^{q, \infty}(\mathbb{R}^n))$, we only need to show continuity at $t=0$ which it is obtained by \eqref{contntyatzero}.
Now, we are going to prove the estimate \eqref{normestimatesimple}.
To this end, we take the $L^q$-norm of both sides of \eqref{solution1} and also consider $q=r$ and $p=1$ in \eqref{estimationws,t}. Then we have
\begin{align}
	\|u(t)\|_q &\leq \|\operatorname{W}_{0,t}u_{0}\|_{q}+\int\limits_0^{t}\|\operatorname{W}_{s,t}f(s)\|_{q}ds\notag\\
	&\leq \|u_{0}\|_{q}+\int\limits_0^{t}\|f(s)\|_{q}ds, \quad\forall t\in \mathbb{R}_{+}.
\end{align}

Here, we aim to demonstrate uniqueness for $1<q<\infty$. Suppose $\tilde{u}$ is another solution to \eqref{cauchypr50}.
 Then,
\begin{equation}\label{cauchypr52}
\begin{cases}
	\partial_t \tilde{u}(t)-\mathcal{L}_t\tilde{u}(t)=f(t),\quad\forall t\in\mathbb{R}_+,\\
	\lim\limits_{t\to0+}\tilde{u}(t)=u_0.
\end{cases}
\end{equation}
Subtracting \eqref{cauchypr50} and \eqref{cauchypr52}, we have
\begin{equation}\label{cauchypr53}
\begin{cases}
	\partial_t \left(u(t)-\tilde{u}(t)\right)-\mathcal{L}_t(u(t)-\tilde{u}(t))=0,\quad\forall t\in\mathbb{R}_+,\\
	\lim\limits_{t\to0+}\left(u(t)-\tilde{u}(t)\right)=0.
\end{cases}
\end{equation}
Applying \eqref{nonincreasingenergy} from Lemma \ref{energyestim} to \eqref{cauchypr53}, we deduce uniqueness as follows:
\begin{equation}
	\|u(t)-\tilde{u}(t)\|_q\le\|u_0-u_0\|_q=0,\quad\forall t\in[0,+\infty),~~\forall 1<q<\infty.
\end{equation}
To prove the final part concerning the positivity of the solution $u$, we employ arguments analogous to those used in Proposition \ref{pr3}.
\end{proof}

We will see bellow that under certain assumptions we have also growth estimates.
\begin{proposition}\label{pr6}
Let $1\le p,q,r<+\infty$ be such that
$
\frac1p+\frac1q=\frac1r+1.
$
In the terminology of Proposition \ref{pr5}, assume that $n(p-1)\alpha<2p$ and
$$
f\in L^\beta((0,t),L^q(\mathbb{R}^n)),~~ \text{where} ~~ \frac1\alpha+\frac1\beta=1, ~~ \text{for} ~~ 1<\alpha<\infty, \quad \forall t\in\mathbb{R}_+.
$$
Moreover, For each 
$t \in\mathbb{R}_+$, let $\lambda_{min} (a(t))$ denote the smallest eigenvalue of the matrix $a(t)$. Now, we define the function
\begin{align}
	F&: \mathbb{R}_+ \longrightarrow \mathbb{R}_+,\notag\\
	F(t)&\doteq \int\limits\limits_0^t \lambda_{\min}(a(\tau))d\tau,\quad \forall t\in \mathbb{R}_{+}.
\end{align}
Then, there exist 
$C\doteq\frac{1}{p^{\frac{n}{2p}}(4\pi)^{\frac{n(p-1)}{2p}}}>0$,
such that
\begin{equation}\label{constant}
	\|u(t)\|_r\le C(F(t))^{-\frac{n(p-1)}{2p}}\|u_0\|_q+C\|\|f(\cdot)\|_q\|_{L^\beta((0,t))}\left(\int\limits\limits_0^t\frac{1}{\left(F(t)-F(s)\right)^{\frac{(p-1)n\alpha}{2p}}}ds\right)^{\frac{1}{\alpha}}.
\end{equation}
In particular, if $\lambda_{\min}(a)$ is 
bounded below by a positive constant $\gamma_0$, then there exists a constant $C>0$ such that
$$
\|u(t)\|_r\le  C\gamma_0 ^{-\frac{n(p-1)}{2p}}\left(t^{-\frac{n(p-1)\alpha}{2p}}\|u_{0}\|_q+\dfrac{t^{\frac{1}{\alpha}-\dfrac{n(p-1)}{2p}}}{\left(1-\frac{n(p-1)\alpha}{2p} \right)^{\frac{1}{\alpha}}} \|\|f(\cdot)\|_{q}\|_{L^\beta((0,t))}\right),\quad \forall t\in\mathbb{R}_{+}.
$$
\end{proposition}
\begin{proof}
We start the proof by the following calculation:

\begin{align}\label{detA(t)-A(s)lambda}
	\det[A(t)-A(s)]&\geq (\lambda_{\min}[A(t)-A(s)])^n =\left(\lambda_{\min}\left[\int\limits_s^t (a(\tau))d\tau\right]\right)^n\notag\\
	&\geq\left(\int\limits_s^t \lambda_{min}(a(\tau))d\tau\right)^n=(F(t)-F(s))^n,\quad\forall (s,t)\in \Delta,
\end{align}
where \eqref{detA(t)-A(s)lambda} is valid based on the following:
\begin{align}\label{lambdamin1}
	\lambda_{\min}\left[\int\limits_s^t a(\tau)d\tau\right]&=\inf_{|\xi|=1}\sum^{n}_{i,j=1}\xi_{i}\xi_{j}\int\limits_s^t a_{i,j}(\tau)d\tau=\inf_{|\xi|=1}\int\limits_s^t \sum^{n}_{i,j=1}\xi_{i}\xi_{j}a_{i,j}(\tau)d\tau\notag\\&\geq \int\limits_s^t \inf_{|\xi|=1} \sum^{n}_{i,j=1}\xi_{i}\xi_{j}a_{i,j}(\tau)d\tau=\int\limits_s^t \lambda_{\min}a(\tau) d\tau, \quad\forall (s,t)\in \Delta.
\end{align}
Raising both sides of \eqref{lambdamin1} to $n$, one can conclude the validity of \eqref{detA(t)-A(s)lambda}.

In particular, setting $s=0$ in \eqref{detA(t)-A(s)lambda}, we obtain  
\begin{align}\label{detA(t)lambda}
	\det[A(t)]&\geq F^n(t),\quad\forall t\in \mathbb{R}_{+}.
\end{align}

We now take the $L^r-norm$ of both sides of \eqref{solution1}. Applying the equations \eqref{detA(t)lambda}, \eqref{detA(t)-A(s)lambda}, and \eqref{normestimatesimple},  together with H\"{o}lder's inequality for $1<\alpha, \beta<\infty$ satisfying $\frac{1}{\alpha}+\frac{1}{\beta}=1$, and proceeding under the assumption that
$n(p-1)\alpha<2p$, we deduce that $	\beta>\frac1{1-\frac{(p-1)n}{2p}}$. Then, using the hypothesis
$
f\in L^\beta((0,t),L^q(\mathbb{R}^n))$, we obtain
\begin{align}
	\|u(t)\|_{r}&\leq \|\operatorname{W}_{0,t}u_{0}\|_r+\int\limits_0^{t}\|\operatorname{W}_{s,t}f(s)\|_{r}ds\notag\\
	&\leq \frac{1}{p^{\frac{n}{2p}}\left((4\pi)^n \det[A(t)]\right)^{\frac{p-1}{2p}}}\|u_{0}\|_{q}
	+\int\limits_0^{t}\frac{1}{p^{\frac{n}{2p}}\left((4\pi)^n\det[A(t)-A(s)]\right)^{\frac{p-1}{2\pi}}}\|f(s)\|_{q}ds\notag\\
	&\leq \frac{1}{p^{\frac{n}{2p}}\left((4\pi)^n(F(t))^n\right)^{\frac{p-1}{2p}}}\|u_{0}\|_{q}
	+\int\limits_0^{t}\frac{1}{p^{\frac{n}{2p}}\left((4\pi)^{n}(F(t)-F(s))^n\right)^{\frac{p-1}{2\pi}}}\|f(s)\|_{q}ds\notag\\
	&\leq C F(t)^{-\frac{n(p-1)}{2p}}\|u_{0}\|_{q}+C
	\left(\int\limits_0^t \frac{1}{(F(t)-F(s))^{\frac{n(p-1)\alpha}{2p}}}ds\right)^{\frac{1}{\alpha}}\left(\int\limits_0^t \|f(s)\|^\beta_{q}ds\right)^{\frac{1}{\beta}}\notag\\
	&=CF(t)^{-\frac{n(p-1)}{2p}}\|u_{0}\|_q+C\left(\int\limits_0^t \frac{1}{(F(t)-F(s))^{\frac{n(p-1)\alpha}{2p}}}ds\right)^{\frac{1}{\alpha}}\|\|f(\cdot)\|_{q}\|_{L^\beta((0,t))}.
\end{align}	
Now, we apply the special condition $\lambda_{\min}(a(t)) >\gamma_0>0$, then we can conclude that \begin{align}\label{detA(t)-A(s)gamma}
	\det[A(t)-A(s)]\geq\left(F(t)-F(s)\right)^n\geq ((t-s)\gamma_0)^n,\quad\forall (s,t)\in \Delta.
\end{align}

In particular, setting $s=0$ in \eqref{detA(t)-A(s)gamma}, we obtain
\begin{align}\label{detA(t)gamma}
	\det[A(t)]&\geq F^n(t)\geq (t\gamma_0)^n,\quad\forall t\in \mathbb{R}_{+}.
\end{align}
Therefore, we conclude that
\begin{align}\label{estimateoperatornorm}
	\|u(t)\|_{r}&\leq \|\operatorname{W}_{0,t}u_{0}\|_r+\int\limits_0^{t}\|\operatorname{W}_{s,t}f(s)\|_{r}ds\notag\\
	&\leq \frac{1}{p^{\frac{n}{2p}}\left((4\pi)^n \det[A(t)]\right)^{\frac{p-1}{2p}}}\|u_{0}\|_{q}
	+\int\limits_0^{t}\frac{1}{p^{\frac{n}{2p}}\left((4\pi)^n\det[A(t)-A(s)]\right)^{\frac{p-1}{2\pi}}}\|f(s)\|_{q}ds\\
	&\leq \frac{1}{p^{\frac{n}{2p}}(4\pi)^{\frac{n(p-1)}{2p}}\gamma_0^{\frac{n(p-1)}{2p}} t^{\frac{n(p-1)}{2p}}}\|u_{0}\|_{q}
	+\int\limits_0^{t}\frac{1}{p^{\frac{n}{2p}}(4\pi)^{\frac{n(p-1)}{2p}}\gamma_0^{\frac{n(p-1)}{2p}} (t-s)^{\frac{n(p-1)}{2p}}}\|f(s)\|_{q}ds\notag\\
	&\leq C\gamma_0 ^{-\frac{n(p-1)}{2p}} t^{-\frac{n(p-1)}{2p}}\|u_{0}\|_{q}+C
	\left(\int\limits_0^t \frac{1}{\gamma_0^{\frac{n(p-1)}{2p}}(t-s)^{\frac{n(p-1)\alpha}{2p}}}ds\right)^{\frac{1}{\alpha}}\left(\int\limits_0^t \|f(s)\|^\beta_{q}ds\right)^{\frac{1}{\beta}}\notag\\
	&=C\gamma_0 ^{-\frac{n(p-1)}{2p}}\left(t^{-\frac{n(p-1)\alpha}{2p}}\|u_{0}\|_q+\dfrac{t^{\frac{1}{\alpha}-\dfrac{n(p-1)}{2p}}}{\left(1-\frac{n(p-1)\alpha}{2p} \right)^{\frac{1}{\alpha}}} \|\|f(\cdot)\|_{q}\|_{L^\beta((0,t))}\right),\quad \forall t\in\mathbb{R}_{+}.\notag
\end{align}	
\end{proof}

\section{Stability}
In this section, we study the following proposition concerning the stability of the solution maps. We begin with the case of $\mathcal{S}(\mathbb{R}^n)$.
\begin{proposition}\label{pr7}
	The solution maps, homogeneous
	\begin{align}\label{firsthomogeneous}
		\mathrm{W}_{0,\natural}&:\mathcal{S}(\mathbb{R}^n)\longrightarrow C^\infty(\mathbb{R}_+,\mathcal{S}(\mathbb{R}^n))\,\cap\,C([0,+\infty),\mathcal{S}(\mathbb{R}^n)),\notag\\
		\mathrm{W}_{0,\natural}&h(t)=\mathrm{W}_{0,t}h,\quad\forall t\in[0,+\infty), ~\forall h\in \mathcal{S}(\mathbb{R}^n),
	\end{align}
	and inhomogeneous
	\begin{align}\label{firstinhomoeneous}
		\int\limits_0^\natural \mathrm{W}_{\sharp,\natural} :~ 
		&C^\infty(\mathbb{R}_+, \mathcal{S}(\mathbb{R}^n))
		\cap C([0,+\infty), \mathcal{S}(\mathbb{R}^n))) \notag\\
		&\longrightarrow
		C^\infty(\mathbb{R}_+, \mathcal{S}(\mathbb{R}^n))
		\cap C([0,+\infty), \mathcal{S}(\mathbb{R}^n)),\notag\\
		\int\limits_0^\natural&\mathrm{W}_{\sharp,\natural}h(t)=\int\limits_0^t\mathrm{W}_{s,t}h(s)ds,\quad\forall t\in[0,+\infty),~\forall h \in C^\infty(\mathbb{R}_+, \mathcal{S}(\mathbb{R}^n))
		\cap C([0,+\infty), \mathcal{S}(\mathbb{R}^n))),
	\end{align}
	are continuous and linear.
\end{proposition}
\begin{proof}
	To prove continuity and linearity of \eqref{firsthomogeneous} and \eqref{firstinhomoeneous}, it suffices to first establish these properties for the following maps:
	\begin{itemize}
		\item [i.]
		\begin{align}\label{i}
			\mathrm{W}_{0,\natural}&:\mathcal{S}(\mathbb{R}^n)\longrightarrow C^\infty(\mathbb{R}_+,\mathcal{S}(\mathbb{R}^n)),\notag\\
			&\mathrm{W}_{0,\natural}h(t)=\mathrm{W}_{0,t}h,\quad\forall t\in[0,+\infty), ~\forall h\in \mathcal{S}(\mathbb{R}^n),
		\end{align}
		and
		\item[ii.]
		\begin{align}\label{ii}
			\mathrm{W}_{0,\natural}&:\mathcal{S}(\mathbb{R}^n)\longrightarrow C([0,+\infty),\mathcal{S}(\mathbb{R}^n)),\notag\\		&\mathrm{W}_{0,\natural}h(t)=\mathrm{W}_{0,t}h,\quad\forall t\in[0,+\infty), ~\forall h\in \mathcal{S}(\mathbb{R}^n),
		\end{align}
		as well as
		\item [iii.]
		\begin{align}\label{iii}
			&\int\limits_0^\natural\mathrm{W}_{\sharp,\natural}:C^\infty(\mathbb{R}_+,\mathcal{S}(\mathbb{R}^n))\cap\,C([0,+\infty),\mathcal{S}(\mathbb{R}^n))\longrightarrow C^\infty(\mathbb{R}_+,\mathcal{S}(\mathbb{R}^n)),\notag\\			&\int\limits_0^\natural\mathrm{W}_{\sharp,\natural}h(t)=\int\limits_0^t\mathrm{W}_{s,t}h(s)ds,\quad\forall t\in[0,+\infty), ~\forall h\in C^\infty(\mathbb{R}_+,\mathcal{S}(\mathbb{R}^n))\cap\,C([0,+\infty),\mathcal{S}(\mathbb{R}^n)),
		\end{align}
		finally
		\item[iv.]
		\begin{align}\label{iv}
			\int\limits_0^\natural\mathrm{W}_{\sharp,\natural}&:C^\infty(\mathbb{R}_+,\mathcal{S}(\mathbb{R}^n))\cap\,C([0,+\infty),\mathcal{S}(\mathbb{R}^n))\longrightarrow C([0,+\infty),\mathcal{S}(\mathbb{R}^n)),\notag\\
			&\int\limits_0^\natural\mathrm{W}_{\sharp,\natural}h(t)=\int\limits_0^t\mathrm{W}_{s,t}h(s)ds,\quad\forall t\in[0,+\infty), ~\forall h\in C^\infty(\mathbb{R}_+,\mathcal{S}(\mathbb{R}^n))\cap\,C([0,+\infty),\mathcal{S}(\mathbb{R}^n)).
		\end{align}
	\end{itemize}
	The linearity of operators \eqref{i}, \eqref{ii}, \eqref{iii}, and \eqref{iv} is immediate.
	To prove the continuity of \eqref{i}, let $K \subset \mathbb{R}_+$ be a compact set and $\alpha, \beta\in\mathbb{N}_0^n$, $\gamma\in\mathbb{N}_0$.
		Since $h\in\mathcal{S}(\mathbb{R}^n)$, using Lemma \ref{avoidingstot}, we have
		\begin{align}\label{ineqstablle1}
			&\left\| \mathrm{W}_{0,\natural} h \right\|_{\alpha,\beta,\gamma, K}= \sup_{t \in K} \sup_{x \in \mathbb{R}^n} \left| x^\alpha \partial_x^\beta \partial_t^\gamma \left( \mathcal{W}(\cdot;0,t) * h  (x)\right)\right|\notag\\
			&\lesssim_{\alpha,\beta,\gamma}\sum_{m \leq \alpha} \sum_{\zeta \leq \alpha-m} 
			\sum_{\mu'=1}^{|m|} \sum_{\mu=1}^\gamma \sup_{t\in K}
			\left(\int_0^t \|a(\tau)\|_{\mathrm{op}} \, d\tau \right)^{\mu'} 
			\Theta_{\mu,\gamma}(t) 
			\sum_{|k'|\le N'}\sum_{|b|\le N}\|h\|_{k',b},
		\end{align}
		where
		$$
		\Theta_{\mu,\gamma}(t)\doteq  \left(\max_{1 \le q \le \gamma} \|\partial_t^{q} A(t)\|_{\mathrm{op}}\right)^{\mu},
		$$
		and $N,N'$ satisfying
		$$
		N > \mu'+\mu+|\beta|-|\zeta| + n, \quad
		N' > |\alpha-m-\zeta|+N +n.
		$$
		Note that since $A\in C^{\infty}(\mathbb{R}_{+},\mathrm{GL}(n))$, the following supremum on compact subset $K$ is finite.
		$$
		\sup_{t\in K} \left(\max_{1 \le q \le \gamma} \|\partial_t^{q} A(t)\|_{\mathrm{op}}\right)^{\mu}\doteq C_{\mu, \gamma, K}<\infty.
		$$
		Therefore, continuing from \eqref{ineqstablle1} and using that $a\in C^{\infty}([0,\infty),\mathrm{GL}(n))$, we obtain
		\begin{align}
			&\left\| \mathrm{W}_{0,\natural} h \right\|_{\alpha,\beta,\gamma, K}\notag\\
			&\lesssim_{\alpha,\beta,\gamma,K}\sum_{m \leq \alpha} \sum_{\zeta \leq \alpha-m} 
			\sum_{\mu'=1}^{|m|} \sum_{\mu=1}^\gamma 
			\left(\sup_{\tau\in[0,\max{K}]}\|a(\tau)\|_{op} \right)^{\mu'} 
			C_{\mu, \gamma, K}
			\sum_{|k'|\le N'}\sum_{|b|\le N}\|h\|_{k',b}\notag\\
			&\lesssim_{\alpha,\beta, \gamma, K}\sum_{|k'|\le N'}\sum_{|b|\le N}\|h\|_{k',b},
		\end{align}
		where $\|\cdot\|_{k',b}$ is a Schwartz seminorm. Hence the proof of \eqref{i} is complete.
	
To prove \eqref{ii}, let $b\in\mathbb{N}$, then for every $\alpha, \beta\in\mathbb{N}_0^n$, we have
	\begin{align}
			&\|\mathrm{W}_{0,\natural} h\|_{\alpha, \beta, b}=\sup_{t\in[0,b]}\left\|\left(\mathrm{W}_{0,\natural} h \right)(t)  \right\|_{\alpha, \beta}=\sup_{t\in[0,b]}\left\|\mathrm{W}_{0,t} h \right\|_{\alpha, \beta}\notag\\
			&=\sup_{t\in[0,b]}\sup_{x\in\mathbb{R}^n}\left|x^\alpha\partial_x^\beta\mathrm{W}_{0,t} h\right|
			=\sup_{t\in[0,b]}\sup_{x\in\mathbb{R}^n}\left|x^\alpha\partial_x^\beta\left( \mathcal{W}(\cdot;0,t) * h  (x)\right)\right|\displaybreak\notag\\
			&=\sup_{t\in[0,b]}\sup_{x\in\mathbb{R}^n}\left|x^\alpha\left( \mathcal{W}(\cdot;0,t) *\partial_x^\beta h  (x)\right)\right|\notag\\
			&\lesssim_{\alpha,\beta,b} \sum_{m\le \alpha} \sum_{\zeta\le \alpha-m} \sum_{\mu'=1}^{|m|} 
			\left(\sup_{\tau\in[0,b]}\| a(\tau)\|_{op}\right)^{\mu'}\sum_{|k'|\le N'}\sum_{|b|\le N}\|h\|_{k',b}
	\end{align}
	for integers $N,N'$ satisfying
	\[
	N > \mu'+\mu+|\beta|-|\zeta| + n, \quad
	N' >| \alpha-m-\zeta|+N + n.
	\]
	It is worth noting that the final estimate follows from \eqref{lemma22} in Lemma \ref{avoidingstot}, together with the continuity of the matrix function $a$ at zero. Moreover, note that $\|\cdot\|_{k',b}$ is a Schwartz seminorm. Thus, the proof of part $ii$ is complete.
	
	To prove the continuity of \eqref{iii}, let $h\in C^\infty(\mathbb{R}_+,\mathcal{S}(\mathbb{R}^n))\cap\,C([0,+\infty),\mathcal{S}(\mathbb{R}^n))$ and let $K\subset \mathbb{R}_{+}$ be compact. Since $\mathcal{W}(\cdot;s,t) \in \mathcal{S}(\mathbb{R}^n)$, the convolution $\mathcal{W}(\cdot;s,t) \ast h(s)$ gives a Schwartz function and converges locally uniformly in $x$, together with all derivatives with respect to $x$. The equicontinuity of the family 
	$\{\operatorname{W}_{s,t}\}_{(s,t)\in C}$ on compact sets 
	$C \subset \overline{\Delta}$ as stated in Lemma \ref{avoidingstot} ensures that
	\[
	[0,t] \ni s \longmapsto \operatorname{W}_{s,t}h(s) \in \mathcal{S}(\mathbb{R}^n),
	\]
	is continuous.
	By the same equicontinuity, the mapping
	$$
	\overline{\Delta} \cap [0,\max K]^2 \ni (s,t) 
	\longmapsto \operatorname{W}_{s,t}\partial^{\beta}_{x}h(s)
	\in \mathcal{S}(\mathbb{R}^n),
	$$
	is jointly continuous. Hence, by \eqref{lemma21} in Lemma \ref{avoidingstot}, for every fixed $\alpha,\beta\in\mathbb{N}_0^n$ and $\gamma\in\mathbb{N}_0$, we have
	\begin{align}\label{normsigma}
		&\left\|\int\limits_0^\natural\mathrm{W}_{\sharp,\natural}h\right\|_{\alpha,\beta,\gamma,K}=\sup_{t\in K}\sup_{x\in\mathbb{R}^n}\left|x^{\alpha}\partial_{x}^{\beta}\partial_t^\gamma\int\limits_0^t\mathrm{W}_{s,t}h(s,x)ds\right|\notag\\
		&=\sup_{t\in K}\sup_{x\in\mathbb{R}^n}\left|x^{\alpha}\partial_{x}^{\beta}\left(\gamma\partial_t^{\gamma-1}h(t,x)+\int\limits_0^t\partial_{t}^\gamma\mathrm{W}_{s,t}h(s,x)ds\right)\right|\notag\\
		&\leq \sup_{t\in K}\sup_{x\in\mathbb{R}^n}\left|x^{\alpha}\partial_{x}^{\beta}\gamma\partial_t^{\gamma-1}h(t,x)\right|+\sup_{t\in K}\sup_{x\in\mathbb{R}^n}\left|x^{\alpha}\partial_x^{\beta}\int\limits_0^t\partial_{t}^\gamma\mathrm{W}_{s,t}h(s,x)ds\right|\notag\\
		&\leq \gamma\|h\|_{\alpha,\beta,\gamma-1,K}+\sup_{t\in K}\sup_{x\in\mathbb{R}^n}\int\limits_0^t\left|x^{\alpha}\left(\partial_{t}^\gamma\mathcal{W}(\cdot;s,t)\ast \partial_x^{\beta}h(s,\cdot)(x)\right)\right|ds\notag\\
		&\lesssim_{\alpha,\beta,\gamma}\|h\|_{\alpha,\beta,\gamma-1,K}+\sup_{t \in K} \sum_{m \le \alpha} \sum_{\zeta \le \alpha-m} 
		\sum_{\mu'=1}^{|m|} \sum_{\mu=1}^\gamma 
		\left( \int\limits_0^t \|a(\tau)\|_{\mathrm{op}} \, d\tau \right)^{\mu'} 
		\Theta_{\mu,\gamma}(t) \int\limits_{0}^t
		I_{m,\zeta,\mu',\mu,\beta}(h(s)) ds \notag\displaybreak\\
		&\lesssim_{\alpha,\beta, \gamma, n} \|h\|_{\alpha,\beta,\gamma-1,K}\notag\\
		&+\max K  \sum_{m \le \alpha} \sum_{\zeta \le \alpha-m} 
		\sum_{\mu'=1}^{|m|} \sum_{\mu=1}^\gamma 
		\left( \int\limits_0^{\max K} \|a(\tau)\|_{\mathrm{op}} \, d\tau \right)^{\mu'} \sup_{t \in K}
		\Theta_{\mu,\gamma}(t) \, \sum_{|k'|\le N'}\sum_{|b|\le N}\|h(s)\|_{k',b},
	\end{align}
	where integers $N,N'$ satisfy
	\[
	N > \mu'+\mu+|\beta|-|\zeta| + n, \quad
	N' > |\alpha-m-\zeta|+N + n.
	\]
	Thus, the proof of part $iii$ is complete.
	
To prove \eqref{iv}, let $ h\in C^\infty(\mathbb{R}_+,\mathcal{S}(\mathbb{R}^n))\cap\,C([0,+\infty),\mathcal{S}(\mathbb{R}^n))$, and let $b\in \mathbb{N}$. Since $\mathcal{W}(\cdot;s,t) \in \mathcal{S}(\mathbb{R}^n)$, the convolution $\mathcal{W}(\cdot;s,t) \ast h(s)$ gives a Schwartz function and converges locally uniformly in $x$, together with all derivatives with respect to $x$. Therefore, for every $\alpha, \beta\in\mathbb{N}_0^n$, using Lemma~\ref{avoidingstot} and the facts that $a\in C^{\infty}([0,\infty),\mathrm{GL}(n))$ and that $f\in C([0,\infty),\mathcal{S}(\mathbb{R}^{n}))$, we have
\begin{align}
	&\left\|\int\limits_0^\natural\mathrm{W}_{\sharp,\natural}h\right\|_{\alpha,\beta,b}=\sup_{t\in [0,b]}\sup_{x\in\mathbb{R}^n}\left|x^{\alpha}\partial_{x}^{\beta}\int\limits_0^t\mathrm{W}_{s,t}h(s,x)ds\right|\notag\\
			&=\sup_{t\in [0,b]}\sup_{x\in\mathbb{R}^n}  \left|\int\limits_0^t x^{\alpha}\partial^{\beta}_{x}h(s)\ast\mathcal{W}(\cdot;s,t)(x) ds\right|\leq \sup_{t\in [0,b]}\sup_{x\in\mathbb{R}^n}  \int\limits_0^t \left|x^{\alpha}\partial^{\beta}_{x}h(s)\ast\mathcal{W}(\cdot;s,t)(x)\right| ds\notag\\
			&\lesssim_{\alpha,\beta}
			\sum_{m\le \alpha}
			\sum_{\zeta\le \alpha-m}
			\sum_{\mu'=1}^{|m|}\sup_{t\in[0,b]}
			\left(\int\limits_{0}^{t}\| a(\tau)\|_{op}\,d\tau\right)^{\mu'}\, 
			\sum_{|b'|\leq N}\sum_{|k'|\leq N'} 
			\int^{t}_{0}\|h(s)\|_{k', b'}ds\notag\\
			&\lesssim_{b,\alpha}	\sum_{m\le \alpha}
			\sum_{\zeta\le \alpha-m}
			\sum_{\mu'=1}^{|m|}
			\left(\sup_{\tau\in[0,b]}\| a(\tau)\|_{op}\right)^{\mu'}\, 
			\sum_{|b'|\leq N}\sum_{|k'|\leq N'} 
			\sup_{s\in[0,b]}\|h(s)\|_{k', b'}\notag\\
			&\forall \alpha, \beta\in\mathbb{N}_0, N > \mu'+\mu+|\beta|-|\zeta| + n, ~
			N' >| \alpha-m-\zeta|+ N + n,\notag
		\end{align}
		where $\|\cdot\|_{k', b'}$ is our Schwartz seminorm here.
	Thus, the proof of part iv is complete.
\end{proof}
Now, we present the following proposition regarding the stability of the solution maps in $W^{q,\infty}(\mathbb{R}^{n})$.
\begin{proposition} For any $q\in[1,+\infty)$, the solution maps, homogeneous
	\begin{align}\label{homogeneoussolution'1}
		\mathrm{W}_{0,\natural}&:L^q(\mathbb{R}^n)\longrightarrow C^\infty(\mathbb{R}_+,W^{q,\infty}(\mathbb{R}^n))\,\cap\,C([0,+\infty),L^q(\mathbb{R}^n)),\notag\\
		&\mathrm{W}_{0,\natural}h(t)=\mathrm{W}_{0,t}h,\quad\forall t\in[0,+\infty),
	\end{align}
	and inhomogeneous
	\begin{align}\label{inhomogeneoussolution'2}
		&\int\limits_0^\natural\mathrm{W}_{\sharp,\natural}:
		C^\infty(\mathbb{R}_+,W^{q,\infty}(\mathbb{R}^n))\,\cap\,C([0,+\infty),W^{q,\infty}(\mathbb{R}^n))\longrightarrow\notag\\ &C^\infty(\mathbb{R}_+,W^{q,\infty}(\mathbb{R}^n))\,\cap\,C([0,+\infty),L^q(\mathbb{R}^n)),\notag\\
		&\int\limits_0^\natural\mathrm{W}_{\sharp,\natural}h(t)=\int\limits_0^t\mathrm{W}_{s,t}h(s)ds,\quad\forall t\in[0,+\infty),
	\end{align}
	are continuous and linear.
\end{proposition}

\begin{proof}
To prove continuity and linearity of \eqref{homogeneoussolution'1} and \eqref{inhomogeneoussolution'2}, it suffices to first establish these properties for the following maps:
	\begin{itemize}
		\item [i.]
		\begin{align}\label{i1}
			\mathrm{W}_{0,\natural}&:L^{q}(\mathbb{R}^n)\longrightarrow C^\infty(\mathbb{R}_+,W^{q,\infty}(\mathbb{R}^n)),\notag\\
			&\mathrm{W}_{0,\natural}h(t)=\mathrm{W}_{0,t}h,\quad\forall t\in[0,+\infty),
		\end{align}
		\item[ii.]
		\begin{align}\label{ii1}			\mathrm{W}_{0,\natural}&: 	L^{q}(\mathbb{R}^n)\longrightarrow	C([0,+\infty),L^{q}(\mathbb{R}^n)),\notag\\	&\mathrm{W}_{0,\natural}h(t)=\mathrm{W}_{0,t}h,\quad\forall t\in[0,+\infty),
		\end{align}
		and
\item [iii.]
	\begin{align}\label{iii1}
	&\int\limits_0^\natural \mathrm{W}_{\sharp,\natural} :
	C^\infty(\mathbb{R}_+,W^{q,\infty}(\mathbb{R}^n)) \cap\, C([0,+\infty),W^{q,\infty}(\mathbb{R}^n))
				\notag\\
				&\hspace{4.3em}\longrightarrow C^\infty(\mathbb{R}_+,W^{q,\infty}(\mathbb{R}^n)), \notag\\
				&\int\limits_0^\natural \mathrm{W}_{\sharp,\natural}h(t)
				=\int\limits_0^t \mathrm{W}_{s,t}h(s)\,ds,\quad
				\forall t\in[0,+\infty),
		\end{align}
		as well as
		\item[iv.]
		\begin{align}\label{iiii1}	&\int\limits_0^\natural\mathrm{W}_{\sharp,\natural}:C^{\infty}(\mathbb{R}_{+}, W^{q,\infty}(\mathbb{R}^{n}))\cap C([0,+\infty),W^{q,\infty}(\mathbb{R}^n))\notag\\
			&\hspace{4.3em}\longrightarrow C([0,+\infty),L^{q}(\mathbb{R}^n)),\notag\\	&\int\limits_0^\natural\mathrm{W}_{\sharp,\natural}h(t)=\int\limits_0^t\mathrm{W}_{s,t}h(s)ds,\quad\forall t\in[0,+\infty).
		\end{align}
	\end{itemize}
	Linearity of the operators \eqref{i1}, \eqref{ii1}, \eqref{iii1}, and \eqref{iiii1} is immediate. To prove the continuity of the mapping \eqref{i1},
	let $K\subset\mathbb{R}_+$ be a compact set and $q\in[1,\infty)$, $j\in\mathbb{N}_0$ as well as $\ell\in\mathbb{N}_0^n$. Then, since $h\in L^q(\mathbb{R}^n)$,
	considering Remark \ref{notation2} and Young's inequality, we have
	\begin{align}
		&\|\mathrm{W}_{0,\natural}h\|_{q,j,\ell,K} =\notag\\
		&\sup_{t\in K}\|\partial^{j}_{t}\partial_x^{\ell}(\mathrm{W}_{0,\natural}h(t))\|_{q}=\sup_{t\in K}\|\partial^{j}_{t}\partial^{\ell}_{x}\mathrm{W}_{0,t}h\|_{q}
		=\sup_{t\in K}\|\partial^{j}_{t}\partial^{\ell}_{x}(	\mathcal{W}(\cdot;0,t)*h)\|_{q}\notag\\&=	\sup_{t\in K}\|(\partial^{j}_{t}\partial^{\ell}_{x}\mathcal{W}(\cdot;0,t))*h)\|_{q}\leq \sup_{t\in K}\|\partial^{j}_{t}\partial^{\ell}_{x}\mathcal{W}(\cdot;0,t)\|_{1}\|h\|_{q}\lesssim_{j,\ell,K}\|h\|_q.
	\end{align}
	To prove \eqref{ii1}, let $b\in\mathbb{N}$ and $h\in L^q(\mathbb{R}^n)$. Then, by Remark \ref{notation2}, one has
	\begin{align}\label{tnotzero}
		\|\mathrm{W}_{0,\natural}h\|_{q,b}&=\sup_{t\in[0,b]}\left\| \left( \mathrm{W}_{0,\natural}h\right) (t)\right\|_{q}=	\sup_{t\in[0,b]}\left\| \mathrm{W}_{0,t}h\right\| _{q}=\sup_{t\in[0,b]}\left\| \mathcal{W}(\cdot;0,t)*h\right\| _{q}\\
		&\leq 	\sup_{t\in[0,b]}\left\|\mathcal{W}(\cdot;0,t) \right\| _{1}\left\| h\right\| _{q}=\left\| h\right\| _{q}\notag.
	\end{align}
To prove \eqref{iii1},  let $K\subset\mathbb{R}_+$ be a compact set and $q\in[1,\infty)$, $j\in\mathbb{N}_0$ as well as $\ell\in\mathbb{N}_0^n$. Suppose that $h\in C^\infty(\mathbb{R}_+,W^{q,\infty}(\mathbb{R}^n))\cap C([0,+\infty),W^{q,\infty}(\mathbb{R}^n))$, it follows that for every $\beta \in \mathbb{N}_0^n$ 
by Lemma 9.1 from \cite{Brezis} and the argument above the relation \eqref{maximalregularity1}, we have 
$$
\partial_x^\beta (h(s)\ast\mathcal{W}(\cdot;s,t))= (\partial_x^\beta h(s)\ast\mathcal{W}(\cdot;s,t))\quad \forall(s,t)\in\Delta.
$$
Moreover, the map
$
s \mapsto \partial_x^\beta h(s)
$
is continuous from $[0,t]$ into $L^q(\mathbb{R}^n)$. In particular,
$$
\sup_{s\in[0,t]} \|\partial_{x}^\beta h(s)\|_{q} < \infty.
$$ Moreover, the equicontinuity of $\{\operatorname{W}_{s,t}\}_{(s,t)\in C}$ on compact sets $C \subset \overline{\Delta}$, as stated in Lemma \ref{avoidingstot},  implies that
	\[
	[0,t]\ni s \mapsto \operatorname{W}_{s,t} h(s),
	\]
	is continuous. By the same equicontinuity of the family $\{\operatorname{W}_{s,t}\}_{(s,t)\in C}$ it follows that the function $\overline\Delta\cap[0,\max K]^2\ni(s,t)\mapsto\operatorname{W}_{s,t}\partial^{\ell}_{x}h(s)$ is (jointly) continuous, so that we can compute for $j>0$ the quantity
\begin{align}
	&\partial^{j}_{t}\partial_{x}^{\ell}	\int\limits_0^t\mathrm{W}_{s,t}h(s,x)ds=\partial^{j}_{t}\partial_{x}^{\ell}	\int\limits_0^t\mathcal{W}(\cdot;s,t)\ast h(s,\cdot)(x)ds=\partial_t^j\partial^{\ell}_{x}\int\limits_0^t\int\limits_{\mathbb{R}^n}\mathcal{W}(y;s,t)h(s,x-y)dyds\notag\\
	&=\partial_t^j\int\limits_0^t\int\limits_{\mathbb{R}^n}\mathcal{W}(y;s,t)\partial^{\ell}_{x}h(s,x-y)dyds=j\partial_t^{j-1}\partial^{\ell}_{x}h(t,x)+\int\limits_0^t\partial_t^j\int\limits_{\mathbb{R}^n}\mathcal{W}(y;s,t)\partial^{\ell}_{x}h(s,x-y)dyds\notag\\
	&=j\partial_t^{j-1}\partial^{\ell}_{x}h(t,x)+\int\limits_0^t\partial_t^{j-1}\int\limits_{\mathbb{R}^n}\partial_t\mathcal{W}(y;s,t)\partial^{\ell}_{x}h(s,x-y)dyds,\quad\forall t\in(0,\max K),\quad\forall x\in\mathbb{R}^n.
	\end{align}
By using part 2 in Proposition \ref{pr1} and  the fact that $\mathcal{L}_t$ is formally self-adjoint, we obtain
\begin{equation}
	\int\limits_{\mathbb{R}^n}\partial_t\mathcal{W}(y;s,t)\partial^{\ell} _{x}h(s,x-y)dy=\int\limits_{\mathbb{R}^n}\mathcal{W}(y;s,t)\mathcal{L}_t\partial^{\ell}_{x}h(s,x-y)dy.
\end{equation}
Repeating the process $j$ times, we observe 
$$
\partial_t^j\int\limits_0^t\int\limits_{\mathbb{R}^n}\mathcal{W}(y;s,t)\partial^{\ell}_{x}h(s,x-y)dyds=j\partial_t^{j-1}\partial^{\ell}_{x}h(t,x)+\int\limits_0^t\int\limits_{\mathbb{R}^n}\mathcal{W}(y;s,t)(\partial_t+\mathcal{L}_t)^j \partial^{\ell}_{x}h(s,x-y)dyds,
$$
$$
\forall t\in(0,\max K),\quad\forall x\in\mathbb{R}^n.
$$
Hence, based on the above discussion, it follows that
\begin{align}\label{wqinfty}
	\left\|\int\limits_0^\natural\mathrm{W}_{\sharp,\natural}h\right\|_{q,j,\ell,K}  &= \sup_{t\in K} \left\|\partial^{j}_{t}\partial_{x}^{\ell}	\int\limits_0^t\mathrm{W}_{s,t}h(s)ds\right\|_q\notag\\
	&= \sup_{t\in K} \left\|\partial^{\ell}_{x}\partial^{j}_{t}	\int\limits_0^t\int\limits_{\mathbb{R}^{n}}\mathcal{W}(y;s,t) h(s,\cdot-y)dyds\right\|_{q}\notag\\
	&= \sup_{t\in K} \left\|	\int\limits_0^t\int\limits_{\mathbb{R}^{n}}\mathcal{W}(y;s,t) (\partial_{t}+\mathcal{L}_{t})^{j}\partial^{\ell}_{x}h(s,\cdot-y)ds\right\|_{q}\notag\\
	&\leq \sup_{t\in K}	\int\limits_0^t\left\|\mathcal{W}(\cdot;s,t)\ast  (\partial_{t}+\mathcal{L}_{t})^{j}\partial^{\ell}_{x}h(s,\cdot)\right\|_{q}ds\notag\\
	&\leq \sup_{t\in K}	\int\limits_0^t\left\|\mathcal{W}(\cdot;s,t)\right\|_{1}\left\|  (\partial_{t}+\mathcal{L}_{t})^{j}\partial^{\ell}_{x}h(s,\cdot)\right\|_{q}ds\notag\\
	&= \sup_{t\in K}	\int\limits_0^t\left\|  (\partial_{t}+\mathcal{L}_{t})^{j}\partial^{\ell}_{x}h(s,\cdot)\right\|_{q}ds.
\end{align}
Furthermore, there exists a natural number $N_{\ell,j}\in\mathbb{N}$, multi-indices $,\{\nu_{\ell,j,k}\}_{k=1}^{N_{\ell,j}}$ and smooth functions $\{C_{\ell,j,k}^a\}_{k=1}^{N_{\ell,j}}$ depending on $a$ and its derivatives, such that
	\begin{equation}\label{expression}
		(\partial_t+\mathcal{L}_t)^j \partial^\ell_{x} h(s,x)=\sum_{k=1}^{N_{\ell,j}}C_{\ell,j,k}^a(t)\partial_x^{\nu_{\ell,j,k}}h(s,x),\quad\forall s,t\in\mathbb{R}_+,\quad\forall x\in\mathbb{R}^n.
	\end{equation}
	Plugging \eqref{expression} in \eqref{wqinfty}, one can find that
	\begin{align}
		&\left\|\int\limits_0^\natural\mathrm{W}_{\sharp,\natural}h\right\|_{q,j,\ell,K}\leq \sup_{t\in K}	\int\limits_0^t\left\|  (\partial_{t}+\mathcal{L}_{t})^{j}\partial^{\ell}_{x}h(s,\cdot)\right\|_{q}ds\leq  \sup_{t\in K}	\int\limits_0^t\left\| \sum_{k=1}^{N_{\ell,j}}C_{\ell,j,k}^a(t)\partial_x^{\nu_{\ell,j,k}}h(s,\cdot) \right\|_{q}ds\notag\\
		&\leq  \sup_{t\in K}\sum_{k=1}^{N_{\ell,j}}\int\limits_0^t\left|C_{\ell,j,k}^a(t)\right|\left\| \partial_x^{\nu_{\ell,j,k}}h(s,\cdot) \right\|_{q}ds\leq \sum_{k=1}^{N_{\ell,j}} \sup_{t\in K}\left|C_{\ell,j,k}^a(t)\right|\sup_{s\in [0,t]}\left\| \partial_x^{\nu_{\ell,j,k}}h(s,\cdot) \right\|_{q}t\notag\\
		&\leq \max K \sum_{k=1}^{N_{\ell,j}} \sup_{t\in K}\left|C_{\ell,j,k}^a(t)\right|\sup_{s\in [0,t]}\left\| \partial_x^{\nu_{\ell,j,k}}h(s,\cdot) \right\|_{q}\leq  \max K \sum_{k=1}^{N_{\ell,j}} \sup_{t\in K}\left|C_{\ell,j,k}^a(t)\right|\sup_{s\in [0,\max K]}\left\| \partial_x^{\nu_{\ell,j,k}}h(s,\cdot) \right\|_{q}\notag\\
		&\leq\max K \sum_{k=1}^{N_{\ell,j}} \sup_{t\in K}\left|C_{\ell,j,k}^a(t)\right|\|h\|_{q,\nu_{\ell,j,k},[\max K]+1}.	
	\end{align}
	
	Now, to prove \eqref{iiii1}, let $b\in\mathbb{N}$ and  $h\in C^\infty(\mathbb{R}_+,W^{q,\infty}(\mathbb{R}^n))\cap C([0,\infty),W^{q,\infty}(\mathbb{R}^n))$. Then, by Remark \ref{notation2}, one can see that
\begin{align}\label{finallq}
	\left\|\int\limits_0^\natural\mathrm{W}_{\sharp,\natural}h\right\|_{q,b}&= \sup_{t\in[0,b]}\left\| \left( 	\int\limits_0^\natural\mathrm{W}_{\sharp,\natural}h\right) (t)\right\| _{q}
	=\sup_{t\in[0,b]}\left\| \int\limits_0^t\mathrm{W}_{s,t}h(s)ds\right\| _{q}\notag\\
	&\leq \sup_{t\in[0,b]}\int\limits_0^t\left\| \mathrm{W}_{s,t}h(s)\right\| _{q}ds=\sup_{t\in[0,b]}\int\limits_0^t\left\| \mathcal{W}{(\cdot;s,t)}*h(s)\right\| _{q}ds\notag\\
	&\leq \sup_{t\in[0,b]}\int\limits_0^t\left\| \mathcal{W}{(\cdot;s,t)}\right\| _{1}\left\| h(s)\right\|_{q}ds= \sup_{t\in[0,b]}\int\limits_0^t\left\| h(s)\right\|_{q}ds\notag\\
	&\leq b\sup_{s\in[0,b]}\left\| h(s)\right\| _{q}=b\left\| h\right\|_{q,b}.
\end{align}
\end{proof}

\section*{Acknowledgements}
Z. Avetisyan and M. Ruzhansky acknowledge the support of  the Methusalem programme of the Ghent University Special Research Fund (BOF) (Grant number 01M01021). Z. Avetisyan's work is funded by the FWO Senior Research Grant G022821N.
Z. Keyshams and M. Mikaeili Nia acknowledge the support by the Higher Education and Science Committee, in the frames of the
project grant 25FAST-1A006.

\end{document}